\documentclass[12pt,reqno]{amsart}

\usepackage{amssymb}
\usepackage[dvipsnames]{xcolor}
\usepackage{a4wide}
\usepackage[colorinlistoftodos,prependcaption,textsize=tiny]{todonotes}
\usepackage[all,cmtip]{xy}

\usepackage{graphicx}
\usepackage{setspace}
\usepackage{tikz}
\usepackage{tkz-graph}
\usetikzlibrary{arrows}

  \makeatletter
  \CheckCommand*\refstepcounter[1]{\stepcounter{#1}%
      \protected@edef\@currentlabel
       {\csname p@#1\endcsname\csname the#1\endcsname}%
  }
  \renewcommand*\refstepcounter[1]{\stepcounter{#1}%
    \protected@edef\@currentlabel
      {\csname p@#1\expandafter\endcsname\csname the#1\endcsname}%
  }
  \def\labelformat#1{\expandafter\def\csname p@#1\endcsname##1}
  \DeclareRobustCommand\Ref[1]{\protected@edef\@tempa{\ref{#1}}%
     \expandafter\MakeUppercase\@tempa
  }
  \makeatother

  \makeatletter
  \newcommand{\numberlike}[2]{%
     \expandafter\def\csname c@#1\endcsname{%
         \expandafter\csname c@#2\endcsname}%
  }
  \makeatother

  \def\DefaultNumberTheoremWithin{section}

  \theoremstyle{plain}
  \newtheorem{Lemma}{Lemma}
     \numberwithin{Lemma}{\DefaultNumberTheoremWithin}
     \labelformat{Lemma}{Lemma~#1}
  
     \numberwithin{Claim}{\DefaultNumberTheoremWithin}
     \numberlike{Claim}{Lemma}
     \labelformat{Claim}{Claim~#1}
  \newtheorem{Theorem}{Theorem}
     \numberwithin{Theorem}{\DefaultNumberTheoremWithin}
     \numberlike{Theorem}{Lemma}
     \labelformat{Theorem}{Theorem~#1}
  \newtheorem{Corollary}{Corollary}
     \numberwithin{Corollary}{\DefaultNumberTheoremWithin}
     \numberlike{Corollary}{Lemma}
     \labelformat{Corollary}{Corollary~#1}
  \newtheorem{Proposition}{Proposition}
     \numberwithin{Proposition}{\DefaultNumberTheoremWithin}
     \numberlike{Proposition}{Lemma}
     \labelformat{Proposition}{Proposition~#1}
  \newtheorem{Conjecture}{Conjecture}
     \numberwithin{Conjecture}{\DefaultNumberTheoremWithin}
     \numberlike{Conjecture}{Lemma}
     \labelformat{Conjecture}{Conjecture~#1}

  \theoremstyle{definition}
  \newtheorem{Definition}{Definition}
     \numberwithin{Definition}{\DefaultNumberTheoremWithin}
     \numberlike{Definition}{Lemma}
     \labelformat{Definition}{Definition~#1}

  \theoremstyle{definition}
  
     \numberwithin{Question}{\DefaultNumberTheoremWithin}
     \numberlike{Question}{Lemma}
     \labelformat{Question}{Question~#1}

  \theoremstyle{definition}
  
     \numberwithin{Problem}{\DefaultNumberTheoremWithin}
     \numberlike{Problem}{Lemma}
     \labelformat{Problem}{Problem~#1}

  \theoremstyle{remark}
  
     \numberwithin{Remark}{\DefaultNumberTheoremWithin}
     \numberlike{Remark}{Lemma}
     \labelformat{Remark}{Remark~#1}
  \theoremstyle{remark}

  \newtheorem{Example}{Example}
     \numberwithin{Example}{\DefaultNumberTheoremWithin}
     \numberlike{Example}{Lemma}
     \labelformat{Example}{Example~#1}
  
     \labelformat{Case}{Case~#1}
     \numberwithin{Case}{Lemma}
  
     \labelformat{Step}{Step~#1}
     \numberwithin{Step}{Lemma}

  \labelformat{equation}{(#1)}
  \labelformat{figure}{Figure~#1}
  \labelformat{section}{Section~#1}
  \labelformat{subsection}{Section~#1}
  \labelformat{subsubsection}{Section~#1}

  \def\eqref{\ref}

  \def\Ker{\mathrm{Ker}}
  \def\Im{\mathrm{Im}}

  \def\sgn{\mathrm{sign}}
  \def\rank{\mathrm{rank}\,}

  \def\Hom{\mathcal{H}}
  \def\Path{{\normalfont {\textrm{Path}}}}
  \def\Cube{{\normalfont {\textrm{Cube}}}}
  \def\Sing{{\normalfont {\textrm{Sing}}}}
  \def\Clique{{\normalfont {\textrm{Clique}}}}
  \def\CC{{\mathcal{C}}}

  \def\LL{{\mathcal{L}}}
   \def\DD{{\mathcal{D}}}

\begin{document}

\title
{Discrete Cubical and Path Homologies of Graphs}

\author[Barcelo]{H\'{e}l\`{e}ne Barcelo}

\address[H\'{e}l\`{e}ne Barcelo]
{
The Mathematical Sciences Research Institute, 17 Gauss Way, Berkeley, CA 94720, USA
}
\email{hbarcelo@msri.org}

\author[Greene]{Curtis Greene}

\address[Curtis Greene]
{
Haverford College, Haverford, PA 19041, USA
}
\email{cgreene@haverford.edu}

\author[Jarrah]{Abdul Salam Jarrah}

\address[Abdul Jarrah]
{
Department of Mathematics and Statistics, American University of Sharjah, PO Box 26666, Sharjah, United Arab Emirates
}
\email{ajarrah@aus.edu}

\author[Welker]{Volkmar Welker}

\address[Volkmar Welker]
{
Fachbereich Mathematik und Informatik, Philipps-Universit\"at, 35032 Marburg, Germany
}
\email{welker@mathematik.uni-marburg.de}

\thanks{This material is based upon work supported by the National Science Foundation under Grant No. DMS-1440140 while the authors were in residence at the
Mathematical Sciences Research Institute in Berkeley, California, USA}

\begin{abstract}
    In this paper we study and compare two homology theories for
(simple and undirected) graphs.
The first, which was developed by Barcelo, Caprano, and White, is based on graph maps from hypercubes 
to the graph. The second theory was developed by Grigor'yan, Lin, Muranov, and Yau, and
is based on paths in the graph. Results in  both settings imply that the respective homology groups are isomorphic in homological dimension one. 
We show that, for several infinite classes of graphs, the two theories lead to
isomorphic homology groups in all dimensions. However, we provide an example
for which the homology groups of the two theories are not isomorphic at least in dimensions two and 
three. We establish a natural map from the cubical to the path 
homology groups which is an isomorphism in
dimension one and surjective in dimension two. Again our example shows 
that in general the map is not surjective in dimension three and not 
injective in dimension two. In the process we develop tools to compute the homology groups for both theories in all dimensions.

\end{abstract}
\date{\today}

\maketitle

\section{Introduction}

For a simple finite undirected graph $G$, we study a discrete cubical singular homology 
theory $\Hom_\bullet^\Cube(G)$. This theory is a special case of the discrete cubical 
homology theory $DH_{\bullet,r}(X)$ that was  defined by Barcelo, Caprano and White 
\cite{BCW} for any metric space $X$ and any real number $r > 0$. Their work builds on a discrete homotopy theory for undirected graphs introduced earlier by Barcelo, Kramer, Laubenbacher, and Weaver in \cite{BKLW}. Later work by Babson, Barcelo, de Longueville, and Laubenbacher \cite{BBLL} connects this theory to classical homotopy theory of cubical sets and asks for a corresponding
homology theory. The homology theory
developed in \cite{BCW} is an answer to that question. 
A more general but closely related homotopy theory for directed graphs was developed by Grigor'yan, Lin, Muranov, and Yau  
in \cite{GLMY2}, which also introduces a corresponding homology theory based on directed paths. 
The homotopy theories in \cite{BBLL} and \cite{GLMY2} are identical when $G$ is undirected and from \cite{BCW} and \cite{GLMY2} it follows that the homology theories yield isomorphic homology groups in dimension $1$. In this paper we explore both the similarities and differences between the two homology theories, showing that they agree in all dimensions for many infinite classes of undirected graphs but disagree in general. Both theories differ markedly from classical singular/simplicial homology of graphs 
seen as $1$-dimensional complexes or their clique complexes. For example, when $G$ is a $4$-cycle, both the cubical and path homologies are trivial in all dimensions greater than zero.

In \ref{sec-definitions}, following \cite{BCW} and \cite{GLMY2}, we give precise definitions of both cubical and path homology for undirected graphs, and discuss the differences between these theories and classical simplicial homology of a graph (as a $1$-dimensional simplicial complex as well as of the clique complex of the graph). 
In \ref{sec-homotopyhomology} we give proofs that both cubical and path homology are preserved under homotopy equivalence, along lines that essentially appear in \cite{BCW} and \cite{GLMY2}. These results are used in \ref{sec-computations} to compute homology for a large number of examples, showing in the process that cubical and path homology agree in all of these cases. \ref{sec-map} constructs a natural homomorphism from $\Hom_\bullet^\Cube(G)$ to $\Hom_\bullet^\Path(G)$. We show that the homomorphism is an
isomorphism in dimension $0$ and $1$ and surjective in dimension $2$, hence fueling  speculation that this might explain the isomorphisms observed in \ref{sec-computations}. However, 
\ref{sec-counter} gives a counterexample:  a graph $G$ for which 
$\Hom_\bullet^\Cube(G) \not\cong \Hom_\bullet^\Path(G)$, and for this example the map defined in \ref{sec-map} is neither injective nor surjective.  \ref{sec-questions} suggests several natural questions for further study.

\section{Background: Discrete Homology of Graphs}
\label{sec-definitions}


Throughout the paper let $R$ denote a commutative ring with unit which shall be the ring of coefficients. 
For any positive integer $n$, let  
$[n] := \{1,\dots,n\}$. For graph theory definitions and terminology we refer the reader to \cite{Diestel00}.

\subsection*{Discrete Cubical Homology}

\begin{Definition}
  For $n \geq 1$, the discrete $n$-cube $Q_n$ is the graph whose vertex set $V(Q_n)$ is 
  $\{0,1\}^n := \{(a_1,\dots, a_n)\,|\,a_i \in \{0,1\} \mbox{ for all } i \in [n]\}$, 
  with an edge between two vertices $a$ and $b$ if and only if their Hamming distance 
  is exactly one, that is, there exists $i \in  [n]$ such that $a_{i} \neq b_{i}$ and 
  $a_{j} = b_{j}$ for all $j \neq i$. 
  For $n=0$, we define $Q_0$ to be the $1$-vertex graph with no edges.
\end{Definition}

\begin{Definition}
  Let $G$ and $H$ be simple graphs, i.e. undirected graphs without loops or multiple edges. A graph homomorphism $\sigma: G \longrightarrow H$ is a 
  map from $V(G)$ to $V(H)$ such that, if $\{a,b\}\in E(G)$ then either $\sigma(a)=\sigma(b)$ or $\{\sigma(a),\sigma(b)\}\in E(H)$.
\end{Definition}

\begin{Definition}
  Let $G$ be a simple graph, a graph homomorphism $\sigma: Q_n \longrightarrow G$ is
  called a \textit{singular} $n$-cube on $G$.
\end{Definition}

For each 
$n \geq 0$, let $\LL^\Cube_n(G)$ be the free $R$-module generated by all singular 
$n$-cubes on $G$.
For $n \geq 1$ and each $i \in [n]$, we define two face maps $f_i^+$ and $f_i^-$ from 
$\LL^\Cube_n(G)$ to $\LL^\Cube_{n-1}(G)$ such that, 
for $\sigma \in \LL^\Cube_n(G)$ and $(a_1,\dots,a_{n-1}) \in Q_{n-1}$:
\begin{eqnarray*}
  f_{i}^{+}\sigma(a_1,\dots,a_{n-1})&:=&\sigma(a_1,\dots,a_{i-1},1,a_i,\dots,a_{n-1}),  \\
  f_{i}^{-}\sigma(a_1,\dots,a_{n-1})&:=& \sigma(a_1,\dots,a_{i-1},0,a_i,\dots,a_{n-1}).
\end{eqnarray*}
For $n \geq 1$, a singular $n$-cube $\sigma$ is called \textit{degenerate} 
if $f_{i}^{+}\sigma = f_{i}^{-}\sigma$, for some $i \in [n]$. Otherwise, 
$\sigma$ is called \textit{non-degenerate}.  By definition every
$0$-cube is non-degenerate.


For each $n\geq 0$, let $\DD^\Cube_n(G)$ be the 
$R$-submodule of 
$\LL^\Cube_n(G)$ that is generated by all degenerate singular $n$-cubes, 
and let $\CC^\Cube_n(G)$ be the free $R$-module 
$\LL^\Cube_n(G)/\DD^\Cube_n(G)$, whose elements are called $n$-chains.
Clearly, the cosets of non-degenerate $n$-cubes freely generate $\CC^\Cube_n(G)$.

Furthermore, for each $n\geq 1$, define the boundary operator 
\[
\partial_n^\Cube :\LL^\Cube_n(G) \longrightarrow \LL^\Cube_{n-1}(G)
\]
such that, for each singular $n$-cube $\sigma$,
\[
\partial_n^\Cube(\sigma) = \sum_{i=1}^n (-1)^i \big( f_i^-\sigma - f_i^+\sigma \big)
\]
and extend linearly to all chains in $\LL^\Cube_n(G)$.
When there is no danger of confusion, we will 
abbreviate $\partial^\Cube_n$ as $\partial_n$. 
If one sets $\LL^\Cube_{-1} (G) =\DD^\Cube_{-1} (G)= (0)$ then one can define
$\partial_0^\Cube$ as the trivial map from $\LL_0^\Cube(G)$ to
$\LL_{-1}^\Cube(G)$.

It is easy to check that, for $n\geq 0$,
$\partial_n [\DD^\Cube_n(G)] \subseteq \DD^\Cube_{n-1}(G)$ 
and  $\partial_{n} \partial_{n+1} \sigma= 0$ 
(see \cite{BCW}). 
Hence,  
$\partial_n: \CC^\Cube_n(G) \longrightarrow \CC^\Cube_{n-1}(G)$ is a boundary operator,
and  $\CC^\Cube(G) = (\CC_\bullet^\Cube,\partial_\bullet)$ is a 
chain complex of free $R$-modules. 

\begin{Definition}
  For $n \geq 0$, denote by $\Hom^\Cube_n(G)$ the $n$\textsuperscript{th} 
  homology group of the chain complex $\CC^\Cube(G)$. In other words, 
  $\Hom^\Cube_n(G) := \Ker\,\partial_n/\Im\,\partial_{n+1}$.
\end{Definition}

We represent singular $n$-cubes $\sigma: Q_n \to G$ by sequences of length $2^n$, where
the $i$\textsuperscript{th} term is the value of $\sigma$ on the $i$\textsuperscript{th} vertex, and the vertices of $Q^n$ are indexed in colexicographic order. For example, if 
$G$ is defined as in \ref{fig2}, then the sequence
$(1,2,2,1,2,3,3,2)$ represents the singular $3$-cube with labels as illustrated in \ref{fig1}.
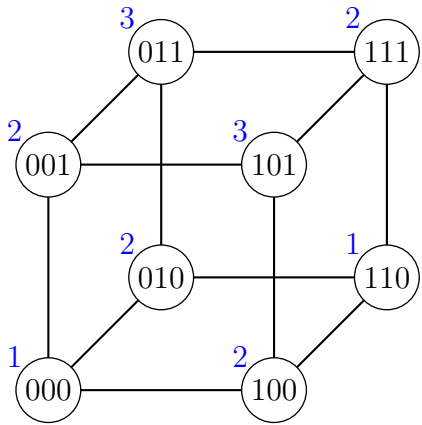
\begin{figure}
\begin{center}
\begin{tikzpicture}[scale=1.5]
   \Vertex[x=0 ,y=0,L=$000$]{1}
      \Vertex[x=2 ,y=0,L=$100$]{2}
         \Vertex[x=1 ,y=1,L=$010$]{3}
            \Vertex[x=3 ,y=1,L=$110$]{4}
               \Vertex[x=0 ,y=2,L=$001$]{5}
                  \Vertex[x=2 ,y=2,L=$101$]{6}
                     \Vertex[x=1 ,y=3,L=$011$]{7}
                        \Vertex[x=3 ,y=3,L=$111$]{8}
      \Edge(1)(2)  \Edge(1)(3)  \Edge(2)(4)
      \Edge(3)(4)  \Edge(5)(6)\Edge(5)(7)\Edge(6)(8)\Edge(7)(8)
      \Edge(1)(5)\Edge(2)(6)\Edge(3)(7)\Edge(4)(8)
      \node [color=blue] at ([shift={(-.3,.3)}]1) {1};
         \node  [color=blue] at ([shift={(-.3,.3)}]2) {2};
            \node  [color=blue] at ([shift={(-.3,.3)}]3) {2};
               \node  [color=blue] at ([shift={(-.3,.3)}]4) {1};
                  \node  [color=blue] at ([shift={(-.3,.3)}]5) {2};
                     \node  [color=blue] at ([shift={(-.3,.3)}]6) {3};
                        \node  [color=blue] at ([shift={(-.3,.3)}]7) {3};
                           \node  [color=blue] at ([shift={(-.3,.3)}]8) {2};
\end{tikzpicture}
\end{center}
\caption{Singular $3$-cube represented by $(1,2,2,1,2,3,3,2)$.}
\label{fig1}
\end{figure}
We represent each coset in $\CC_n^\Cube(G)$ by the unique coset representative in which all terms are non-degenerate.

\begin{Example} Let $G$ be a $4$-cycle, with vertices labeled cyclically, as illustrated in \ref{fig2}.
\begin{figure}
\begin{center}
\begin{tikzpicture}[scale=1.5]
   \Vertex[x=0 ,y=0]{1}
      \Vertex[x=1 ,y=0]{2}
         \Vertex[x=1 ,y=1]{3}
            \Vertex[x=0 ,y=1]{4}
      \Edge(1)(2)  \Edge(2)(3)  \Edge(3)(4)
      \Edge(4)(1)  
\end{tikzpicture}
\caption{Graph $G =$ $4$-cycle.}
\label{fig2}
\end{center}
\end{figure}
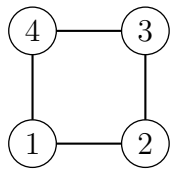
Then
\begin{gather*}
\CC_0^\Cube = \big\langle (1),(2),(3),(4) \big\rangle \\
\CC_1^\Cube = \big\langle (1,2), (2,1), (2,3), (3,2), (3,4), (4,3), (4,1), (1,4) \big\rangle \\
\CC_2^\Cube = \big\langle   (1,1,1,2), (1,1,1,4), \dots<60 \textrm{ more}>\dots,    (4,4,4,1), (4,4,4,3)     \big\rangle. \end{gather*}
The matrix of $\partial_1$ with respect to the above bases is the standard vertex-directed edge incidence matrix of the corresponding directed graph in which each edge is replaced by directed edges in both directions. An easy computation shows that $\partial_1$ has rank $|V|-1 = 4-1 = 3$  and nullity  $8-3 = 5$.  Cycles in $\CC_1^\Cube(G)$ correspond to circulations in $G$,
that is, weighted sums of edges in which the net flow out of each vertex equals zero.
A basis of $\CC_1^\Cube(G)$ may be obtained from any directed cycle basis of $G$, e.g., for the graph $G$ defined in \Ref{fig2} we may take $(1,2)+(2,1), (2,3)+(3,2), (3,4)+(4,3), (4,1)+(1,4)$, and $ (1,2)+(2,3)+(3,4)+(4,1)$.  Each of these $1$-cycles is the boundary of a $2$-chain:
\begin{gather*}
(1,2)+(2,1) = \partial_2((2,1,2,2))\\ (2,3)+(3,2) = \partial_2((3,2,3,3))\\ (3,4)+(4,3) = \partial_2((4,3,4,4))\\ (4,1)+(1,4) = \partial_2((1,4,1,1))\\
(1,2)+(2,3)+(3,4)+(4,1) = \partial_2((1,2,4,3) + (3,4,3,3) + (1,4,1,1)). 
\end{gather*}
Hence $\Hom_0^\Cube(G)= R$ and $\Hom_1^\Cube(G) = (0)$. By somewhat tedious computations one can also show that 
$\Hom_2^\Cube(G)=(0)$. Here we have $\rank(\CC_2^\Cube(G))=64$ and $\rank(\CC_3^\Cube(G))=2432$, and for higher dimensions the problem of computing $\Hom_n^\Cube(G)$ becomes increasingly more difficult.   Fortunately, we are able to prove  more general results (in \ref{sec-computations}) implying that $\Hom_n^\Cube(G)=(0)$ for all $n>0$, for the graph $G$ defined above in \ref{fig2}. 

\end{Example}

\subsection*{Discrete Path Homology}
In a series of papers \cite{GLMY1,GLMY2,GLMY3} a (co)homology and a homotopy theory
for directed graphs are developed. In these theories, an undirected graph is 
interpreted as the directed graph, with each undirected edge replaced by
two oppositely directed edges between its endpoints. 
It is shown in \cite[Thm. 4.22]{GLMY2} that the first homology group
of a directed graph is the abelianization of its fundamental group, where
both homology and homotopy groups are taken in the sense of \cite{GLMY1,GLMY2}.

We now recall the homology theory from \cite{GLMY1}, confining ourselves
to the setting of simple (undirected) graphs.

Let $V$ be a finite set. For $n\ge 0$ we denote by $\LL_n^\Path(V) $ the $R$-module freely generated by 
the set of all $(n+1)-$tuples $(v_0, \dots, v_n)$ of elements in $V$. For each $n$, let $\DD_n^\Path(V)$ denote the submodule generated by
the {\em degenerate} $n$-tuples $(v_0,\dots,v_n)$ where $v_i=v_{i+1}$ for some $i$.  For $n\ge 1$, let
$\partial_n^\Path:\LL_n^\Path(V) \to \LL_{n-1}^\Path(V) $ be defined by
\[
(v_0,\dots,v_n) \longmapsto \sum_{i=0}^n (-1)^i (v_0,\dots,\widehat{v_i},\dots,v_n).
\]
If we set $\LL_{-1}^\Path(V) = \DD_{-1}^\Path(V) = (0)$ we can also define $\partial_{0}^\Path$ as the trivial map from 
$\LL_0^\Path(V)$ to $\LL_{-1}^\Path(V)$.
Again, we will write $\partial_n^\Path = \partial_n$ when there is no ambiguity.
 For $n\ge 0$, it is easy to verify that $\partial_{n} \partial_{n+1} = 0$ and 
$\partial_n[\DD_n^\Path(V)] \subseteq \DD_{n-1}^\Path(V)$. Hence if we define a  sequence of quotients 
\[
\CC_n^\Path(V) = \frac{\LL_n^\Path(V) }{\DD_n^\Path(V) }, \; n = -1,0,1,\dots,
\]
then $\CC^\Path (V) = (\CC^\Path_\bullet(V), \partial_\bullet^\Path)$ 
forms a chain complex.

Now let  $G=(V,E)$ be a simple  graph. Define $\LL_n^\Path(G) \subseteq \LL_n^\Path(V) $ to be the 
submodule of  $\LL_n^\Path(V) $ spanned by all $(v_0,\dots,v_n)$ such that $\{v_i, v_{i+1}\} \in E(G)$ for all $i<n$. 
Thus, $\LL_n^\Path(G) = \LL_n^\Path(V) $ when $G$ is the complete
graph on vertex set $V$.
For all $n \geq 0$, define
$\widetilde{\CC}_n^\Path(G) \subseteq \CC_n^\Path(V)$ to be the submodule of $\CC_n^\Path(V)$ generated by cosets of the form
\[
(v_0, \dots, v_n) + \DD_n^\Path (V),
\]
where $(v_0,\dots,v_n) \in\LL_n^\Path(G)$, and set $\widetilde{\CC}_{-1}^\Path(G) = (0)$.  The sequence $\{\widetilde\CC_n^\Path(G) \}_{n\ge -1}$ is not always a chain complex, since boundaries of paths $(v_0,\dots,v_n) \in \LL_n^\Path(G)$ may contain terms that are not paths in $G$. However, if we define, for $n\ge 0$,  
\[
\CC_n^\Path(G) = \partial_n^{-1}[\widetilde{\CC}_{n-1}^\Path(G)]
\]
and $\CC_{-1}^\Path(G) = (0)$, then $\partial_{n}\partial_{n+1} = 0$ immediately implies that 
$\partial_n[\CC_n^\Path(G)] \subseteq \CC_{n-1}^\Path(G)$, and 
$\CC^\Path(G)=(\CC_\bullet^\Path,\partial_\bullet)$ is a chain complex.  

\begin{Definition}
  For $n \geq 0$, denote by $\Hom^\Path_n(G)$ the $n$\textsuperscript{th} 
  homology group of the chain complex $\CC^\Path(G)$. In other words, 
  $\Hom^\Path_n(G) := \Ker\,\partial_n/\Im\,\partial_{n+1}$.
\end{Definition}

We again identify cosets in $\CC_n^\Path(G)$ with their unique representatives whose terms are all non-degenerate.
Using this notation, if $G$ is the $4$-cycle graph in \ref{fig2}, then 

\begin{gather*}
\CC_0^\Path = \big\langle (1),(2),(3),(4) \big\rangle \\
\CC_1^\Path = \big\langle (1,2), (2,1), (2,3), (3,2), (3,4), (4,3), (4,1), (1,4) \big\rangle \\
\CC_2^\Path = \big\langle (1,2,1), (2,1,2), (2,3,2), (3,2,3), (3,4,3), (4,3,4), (4,1,4), (1,4,1), \qquad\qquad \qquad\\\ \qquad\qquad\qquad(1,2,3)-(1,4,3), (2,3,4)-(2,1,4), (3,4,1)-(3,2,1), (4,1,2) - (4,3,2) \big\rangle. 
\end{gather*}

Represented in this notation, the chain groups $\CC_0^\Path(G)$ and $\CC_1^\Path(G)$ are identical to $\CC_0^\Cube(G)$ and $\CC_1^\Cube(G)$. The boundary map $\partial_1$ again has rank $5$ and its kernel is spanned by $(1,2)+(2,1), (2,3)+(3,2), (3,4)+(4,3), (4,1)+(1,4)$, and $(1,2)+(2,3)+(3,4)+(4,1)$. As before, each of these $1$-cycles is the boundary of a $2$-chain:
\begin{gather*}
(1,2)+(2,1) = \partial_2((1,2,1))\\ (2,3)+(3,2) = \partial_2((2,3,2))\\(3,4)+(4,3) = \partial_2((3,4,3))\\ (4,1)+(1,4) = \partial_2((1,4,1))\\
(1,2)+(2,3)+(3,4)+(4,1) = \partial_2(((1,2,3)-(1,4,3))+(3,4,3)+(1,4,1)). 
\end{gather*}
It follows that $\Hom_0^\Path(G)= R$ and $\Hom_1^\Path(G)=(0)$. Again it is possible to prove directly that $\Hom_2^\Path(G) = (0)$, but more general results in \ref{sec-computations} will show that,
for this example, $\Hom_n^\Path(G)=(0)$ for all $n > 0$. 

\subsection*{Classical Homology of a Graph and its Clique Complex}

We mention two other homology theories of graphs that have a substantial presence in the literature.

Given any undirected graph $G$, we may regard $G$ as a $1$-dimensional simplicial complex and compute its singular (or equivalently, simplicial) homology $\Hom^\textrm{Sing}_\bullet (G)$.  It is elementary and classical (e.g., \cite{Massey91}, Chapter 8) that if $G$ is connected, then $\Hom^\textrm{Sing}_0(G) \cong R$, $\Hom^\textrm{Sing}_1(G) \cong R^{|E(G)|-|V(G)|+1}$, and $\Hom^\textrm{Sing}_n(G)\cong(0)$ for $n>1$.  

Given $G$, we may also construct the clique complex $K_G$ of $G$ (also called the flag complex of $G$; see \cite{St}), whose faces are the subsets of $V(G)$ forming cliques, and compute the simplicial (or equivalently, singular) homology $\Hom_\bullet^{\textrm{Clique}}(G)$ of $K_G$.  If $N=\omega(G)$ is the size of the largest clique in $G$, then $\Hom_n^\textrm{Clique}(G) \cong (0)$ for $n>N$, but $\Hom_n^\textrm{Clique}(G)$ can be nonzero for any $n\le N$. 
If $G$ is a $4$-cycle as in \ref{fig2}, then $\Hom^\textrm{Sing}_n(G)$ and $\Hom^\textrm{Clique}_n(G)$ 
are isomorphic for all $n$, but this is not true in general (for example, when $G$ is a $3$-cycle). We note that a theory analogous to $\Hom^\Cube_\bullet(G)$ can be defined by considering chain groups
spanned in dimension $n$ by graph maps 
from the complete graph on $n+1$ vertices to $G$ and differential given by the alternating sum over the restrictions to complete subgraphs on $n-1$ vertices. This theory can be seen to be equivalent to $\Hom_\bullet^\textrm{Clique}(G)$ (see \cite[p. 76]{Mu}). 

\subsection*{Relationships } This paper will explore connections between the two homology theories $\Hom_\bullet^\Cube(G)$ and $\Hom_\bullet^\Path(G)$ defined above. For many classes of graphs $G$ we have, $\Hom^\Cube_\bullet(G) \cong \Hom^\Path_\bullet(G)$, and we will give several more examples of this phenomenon (see especially \ref{sec-computations}). In \ref{sec-map} we define a homomorphism from 
$\CC^\Cube(G)$ to $\CC^\Path(G)$ that may explain  some of these connections. However, $\Hom^\Cube_\bullet $ and  $\Hom^\Path_\bullet$ are not isomorphic in general, and we give an example illustrating this in \ref{sec-counter}.  

Connections with $\Hom^\Sing_\bullet$ and $ \Hom^\Clique_\bullet$ seem to be less close; for example, when $G$ is a $4$-cycle, the 
discrete cubical
and path homologies are trivial in dimension $1$, but the singular and clique homologies are nontrivial. A combination of results in \cite{BBLL} and \cite{BCW} proves that for any graph $G$, $\Hom_1^\Cube(G) \cong \Hom_1^\Sing(K)$, where $K$ is the CW-complex obtained from $G$ by ``filling in'' all of its triangles and quadrilaterals with $2$-cells.  A similar construction in higher dimensions is conjectured in \cite{BBLL} to give the correct higher homotopy groups, 
and the authors have proposed (private communication) that this might also produce a CW-complex $K^*$ such that
$\Hom_n^\Cube(G) \cong \Hom_n^\Sing(K^*)$ when
$n>1$. Note that $K^*$ has cells of arbitrary high dimension for most graphs. 

\section{Homotopy Equivalence Preserves Homology
\label{sec-homotopyhomology}}

This section describes the connection between the graph
homotopy theory introduced in \cite{BBLL} and \cite{GLMY2}
and the cubical and path homologies introduced in \cite{BCW} and \cite{GLMY2}. First we recall several basic definitions. 

\begin{Definition}\label{boxdef} (See \cite{Hammack11})
If $G$ and $H$ are graphs, the 
{\em Cartesian (or box) product} $G\,\Box\, H$ is the
graph whose vertex set is the Cartesian product
set $V(G)\times V(H)$, and whose edges are 
pairs $\{(g_1,h_1),(g_2,h_2)\}$ such that
either $g_1=g_2$ and $\{h_1,h_2\}\in E(H)$ or
$h_1=h_2$ and $\{g_1,g_2\}\in E(G)$.
\end{Definition}

\begin{Definition}\label{homotopy-def}
Suppose that $G$ and $H$ are graphs, and $f$ and $g$ are graph homomorphisms from $G$ to $H$. Then $f$ and $g$ are {\em homotopic} if there exists a graph homomorphism $\Phi$ from $G\,\Box\, I_m$ to $H$, where $I_m$ denotes the $m$-path with vertex set $\{0,1,\dots,m\}$, such
that $\Phi(\bullet,0)=f$ and $\Phi(\bullet,m)=g$.
\end{Definition}

\begin{Definition}
Two simple undirected graphs $G$ and $H$ are
{\em homotopy equivalent} if there exist graph homomorphisms $\phi: G\to H$ and $\theta: H 
\to G$ such that $\theta \phi$ is homotopic 
to $id_G$ and $\phi \theta$ is homotopic to $id_H$. Here $id_G$ and $id_H$ denote the identity maps on $G$ and $H$, respectively.
\end{Definition}

The connection between the discrete homotopy theory in \cite{BBLL} and \cite{GLMY2} and the homology theories 
introduced in \cite{BCW} and \cite{GLMY2} is expressed by the following theorem, which also provides a key  computational tool.

\begin{Theorem}\label{equiv-thm}
Let $G$ and $H$ be simple, undirected graphs. If $G$ and $H$ are homotopically equivalent, then,  for all $n\ge 0$,  
\begin{enumerate}
\item[(i)]
$\Hom^\Cube_n(G)\cong\Hom^\Cube_n(H)$, and
\item[(ii)]
 $\Hom^\Path_n(G)\cong\Hom^\Path_n(H)$.
 \end{enumerate}

\end{Theorem}

For both parts of \Ref{equiv-thm} it suffices to prove that  if $\alpha$ and $\beta$ are homotopically equivalent maps from $G$ to $H$, then $\alpha$ and $\beta$ induce identical maps on homology. For cubical homology, this result is contained in Theorem 3.8(1) of \cite{BCW}, where it is proved for any discrete metric space. Since details of that argument do not appear in \cite{BCW} we provide them here (in the case of graphs) for completeness. For path homology, \Ref{equiv-thm}(ii) is stated and proved explicitly in \cite{GLMY2} (Theorem 3.3(ii)), for directed graphs. We include a sketch of that argument in the undirected case. Both proofs employ a chain homotopy argument, and are
structurally similar.

\begin{proof}[Proof of \Ref{equiv-thm}(i)]
Suppose that $\alpha$ and $\beta$ are homomorphisms from $G$ to $H$, and  $\Phi$ is a homotopy from $\alpha$ to $\beta$ with $\Phi(x,0)=\alpha(x)$ and $\Phi(x,m) = \beta(x)$ for all $x\in V(G)$, as in \ref{homotopy-def}. 
If $\sigma\in \CC^\Cube_n(G)$, let $\Phi(\sigma,j)$ denote the map defined by
$
\Phi(\sigma,j)(q) = \Phi(\sigma(q),j)
$
for all $q\in Q_n$, and define
$\tilde{\alpha}_n, \tilde{\beta}_n: \CC_n^\Cube(G)\to \CC_n^\Cube(H)$ by
\[
\tilde{\alpha}_n(\sigma) = \Phi(\sigma,0), \quad
\tilde{\beta}_n(\sigma) = \Phi(\sigma,m).
\]
 It is straightforward to show that $\tilde{\alpha}$ and $\tilde{\beta}$ are
chain maps, i.e., $\tilde{\alpha}_{n-1}  \partial_n = \partial_n \tilde{\alpha}_n$ and similarly
for $\tilde{\beta}$.
We will construct a sequence of maps $h_n: \CC_n^\Cube(G)\to \CC_{n+1}^\Cube(H)$ such that
\begin{equation}\label{chomotopy1}
\tilde{\beta}_n - \tilde{\alpha}_n = \partial_{n+1} h_n + h_{n-1} \partial_n,
\end{equation}
for all $n$.
In other words, the sequence $\{h_n\}$ defines a chain homotopy between $\{\tilde{\alpha}_n\}$ and
$\{\tilde{\beta}_n\}$. It follows that if $z\in \CC_n^\Cube(G)$ is a cycle, then 
\[
\tilde{\beta}_n(z) - \tilde{\alpha}_n(z) = \partial_{n+1} h_n (z).
\]
In particular,  $\tilde{\beta}_n(z) - \tilde{\alpha}_n(z) \in \Im\,\partial_{n+1}$ and hence $\tilde{\alpha}_n(z)$ 
and $\tilde{\beta}_n(z)$ lie in the same homology class for all $z$, implying that
$\alpha$ and $\beta$ induce the same maps on homology.

Given a singular $n$-cube $\sigma \in \CC_n^\Cube(G)$, the map $h_n(\sigma) \in \CC_{n+1}^\Cube(H)$ is constructed as follows. For $j=1,\dots,m$, let $h_n^{(j)}(\sigma)\in\CC_{n+1}^\Cube(H)$ be the unique labeled $(n+1)$-cube such that
\begin{gather}
    f^+_1 h_n^{(j)}(\sigma) (q) = \Phi(\sigma(q),j)\nonumber \\
     f^-_1 h_n^{(j)}(\sigma) (q) = \Phi(\sigma(q),j-1) \nonumber,
\end{gather}
for all $q\in Q_n$. Finally, define
\begin{equation}
h_n(\sigma) = h_n^{(1)}(\sigma) + \cdots + h_n^{(m)}(\sigma).\nonumber
\end{equation}
It is immediate from the definition of $h_n^{(j)}$ that
\begin{gather}
    f^+_1 h_n^{(m)}(\sigma) (q) =\beta(\sigma(q)) \nonumber\\
     f^-_1 h_n^{(1)}(\sigma) (q) = \alpha(\sigma(q)), \nonumber
\end{gather}
for all $q\in Q_n$. A few moments of reflection show that
for $i=2,\dots, n$, we have
\begin{equation}\label{equiv-identity2}
f^\epsilon_{i} (h_n^{(j)}(\sigma))=h_{n-1}^{(j)}(f_{i-1}^\epsilon \sigma) 
\end{equation}
for  $j \in [m]$ and $\epsilon \in \{-,+\}$. Computing the right hand side of \Ref{chomotopy1}, we
get
\begin{eqnarray}\nonumber
h_{n-1}^{(j)}(\partial_n(\sigma) )&=& h_{n-1}^{(j)}\bigg(\sum_{i=1}^n (-1)^{i} (f^-_i \sigma - f^+_i \sigma)\bigg)\\
&=& \sum_{i=1}^n (-1)^{i} \big(h_{n-1}^{(j)}(f^-_i \sigma) - h_{n-1}^{(j)}(f^+_i \sigma)\big)\label{equiv-RHSsecond2}
\end{eqnarray}
and
\begin{eqnarray}\label{equiv-RHSfirst}
\partial_{n+1} (h_n^{(j)}(\sigma)) &=&\sum_{i=1}^{n+1} (-1)^{i}\big( f^-_i (h_n^{(j)}(\sigma)) - f^+_i (h_n^{(j)}(\sigma)) \big).
\end{eqnarray}
It follows from \Ref{equiv-identity2} that terms $i=1,\dots,n$ in \Ref{equiv-RHSsecond2} are identical to terms $i=2,\dots, n+1$ in
\Ref{equiv-RHSfirst}, but have opposite signs. Hence they cancel, leaving only the first term
in \Ref{equiv-RHSfirst},
and we obtain 
\begin{eqnarray}
\partial_{n+1}h_n^{(j)}(\sigma) + h_{n-1}^{(j)}\partial_n(\sigma) & =&
f_1^+(h_n^{(j)}(\sigma)) - f_1^-(h_n^{(j)}(\sigma))  \nonumber\\
\label{finalidentity}
&=& \Phi(\sigma,j) - \Phi(\sigma,j-1).
\end{eqnarray}
Summing \ref{finalidentity} over $j$ gives 
\begin{eqnarray}
\partial_{n+1}h_n(\sigma) + h_{n-1}\partial_n(\sigma) & =&
  \sum_{j=1}^m \Phi(\sigma,j) - \Phi(\sigma,j-1) \nonumber \\
   &=& \Phi(\sigma,m) - \Phi(\sigma,0) \nonumber\\
   &=&
   \tilde{\beta}(\sigma) - \tilde{\alpha}(\sigma) \label{finalstep}
\end{eqnarray}
as desired, and \Ref{chomotopy1} is proved.

\end{proof}

\begin{proof}[Proof of \ref{equiv-thm}(ii)] 
The proof in \cite{GLMY2} has essentially the same structure as the proof of 
part (i) given above. We will sketch the argument, using similar notation but focusing on the important differences. Again assume that $\alpha, \beta: G\to H$ are graph homomorphisms, with a homotopy $\Phi$ such that $\Phi(x,0) = \alpha(x)$ and $\Phi(x,m) = \beta(x)$ for all $x\in V(G)$. 
It is shown in \cite[Theorem 2.10]{GLMY2}  that $\alpha$ and $\beta$ induce chain maps $\tilde{\alpha}_n$ and $\tilde{\beta}_n$ from 
$\CC_n^\Path(G)$ to $\CC_n^\Path(H)$. As before, the key step in the present proof is to construct a chain homotopy between the sequences $\{\tilde{\alpha}_n\}$ and $\{\tilde{\beta}_n\}$. 


For $\sigma = (v_0,v_1,\dots,v_n) \in \CC_n^\Path(G)$ and $j \in [m]$, define $h_n^{(j)}(\sigma)
\in \CC^\Path_{n+1}(H)$ as follows:
\[
h_n^{(j)}(\sigma) = 
\sum_{k=0}^n (-1)^k (\Phi(v_0,j-1),\dots,
\Phi(v_k,j-1),\Phi(v_k,j),\dots,\Phi(v_n,j)),
\]
and define
\[
h_n(\sigma) = h_n^{(1)}(\sigma) + \cdots + h_n^{(m)}(\sigma).
\]

\smallskip
\noindent
At this point it is essential to check that $h_n^{(j)}(\sigma) \in \CC^\Path_{n+1}(H)$ for all $j$, since not every linear combination of elements of $\tilde{\CC}^\Path_{n+1}(H)$ 
is an element of
$\CC^\Path_{n+1}(H)$.
An  argument proving this fact can be found in \cite[Proposition 2.12]{GLMY2}, and is omitted here. 

The proof is completed by showing that identity \Ref{finalidentity} holds for the maps $h_n$ just defined, exactly as it did in part (i). This argument is technical but straightforward, and is omitted here. With 
\Ref{finalidentity} in hand,  \Ref{finalstep} follows, and we are done.
\end{proof}

\section{Computations of Homology Groups}\label{sec-computations}

With \ref{equiv-thm} in hand, we have tools that will allow us to compute 
$\Hom_\bullet^\Cube(G)$ and $\Hom_\bullet^\Path(G) $ for large classes of graphs. We give
many examples in this section. Most involve {\em deformation retraction}, a special kind of homotopy equivalence that is frequently easy to recognize. 

\begin{Definition}\label{retractx-def}
Let $G$ be a graph, and let $H$ be an induced subgraph of $G$. That is, $V(H) \subseteq V(G)$ and $E(H)$ consists of all edges in $E(G)$ for which both endpoints belong to $V(H).$ 
\begin{enumerate}
    \item
A {\em retraction of $G$ onto $H$} is a graph homomorphism $r:G\to H$ such that
 $r(y) = y$ for all $y\in V(H)$.
\item A {\em deformation retraction of $G$ onto $H$} is a retraction $r:G \to H$ such that $ir$ is homotopic to $id_G$, where $i$ denotes the inclusion map from $H$ to $G$.
\item
A {\em one-step deformation retraction from $G$ to $H$}
is a deformation retraction $r$ for which $m=1$ in the homotopy (\Ref{homotopy-def}) between
$ri$ and $id_G.$ Equivalently, $r$ is a retraction such that $\{x,r(x)\}$ is an edge for all $x\in V(G)$.
\end{enumerate}
\end{Definition}

If $r$ is a deformation retraction from $G$ to $H$, then, since $ri= id_H$, the following lemma is an immediate consequence of $\Ref{equiv-thm}$.

\begin{Lemma}
If $r$ is a deformation retraction from $G$ onto a subgraph $H$, and $i$ denotes the inclusion map from $H$ to $G$, then $r$ and $i$ define a homotopy equivalence between $G$ and $H$. Consequently, 
$\Hom^\Cube_n(G) \cong \Hom^\Cube_n(H)$ and $\Hom^\Path_n(G) \cong \Hom^\Path_n(H)$,  for all $n \geq 0$.
\end{Lemma}

\smallskip

This result immediately gives  several infinite classes of graphs for which  the cubical and path (reduced) homology is trivial in all dimensions.

\begin{Corollary}\label{cor-examples}
If $G$ is a tree, or a complete graph, or a hypercube, then $\Hom_n^\Cube(G)\cong\Hom_n^\Path(G)\cong(0) $ for all $n>0$.
\end{Corollary}

\begin{proof}If $G$ is a tree and $x\in V(G)$ is a leaf connected to a unique vertex $y$, then the map $r: V(G)\to V(G)\backslash\{x\}$ defined by 
\begin{equation}\label{retract-def}
r(v) = 
\begin{cases}
v & v \ne x, \\
y & v = x,
\end{cases}
\end{equation}
is a one-step deformation retraction from $G$ onto the subgraph $G\backslash x$.  If $G$ is a complete graph, $x\in V(G)$ and $y\ne x$ is any other vertex, then \ref{retract-def} again defines a one-step deformation retraction from $G$ to $G\backslash x$.  If $G$ is a hypercube of dimension $n$, then the map $r$ defined by collapsing any facet onto its opposite facet is a one-step deformation retraction onto a hypercube of dimension $n-1$.  In all three cases, the process can be
repeated, eventually showing that the homology (both cubical and path) is the same as that of a graph with a single vertex.
\end{proof}

These arguments can be extended to a larger class of examples:

\begin{Theorem} Let $G$ be a graph, and $K_1$ and $K_2$ are induced nonempty subgraphs of $G$
such that $V(G) = V(K_1)\cup V(K_2)$ and $V(K_1)\cap V(K_2) = \emptyset$. Suppose there exist vertices $a\in V(K_1)$
and $b\in V(K_2)$ such that $\{a,b\}\in E(G)$, every vertex in $K_1$ is connected to $b$, and every vertex in $K_2$ is connected to $a$. Then $\Hom_n^\Path(G)
\cong \Hom_n^\Cube(G) \cong (0)$ for $n> 0$.
\end{Theorem}

\begin{proof}
Let $H$ be the subgraph of $G$ with vertices $a$ and $b$ and the single edge $\{a,b\}$. Define $r: V(G)\to V(H)$ by
\[
r(x) = \begin{cases}
a & \textrm{if $x \in K_2-\{b\}$},\\
b & \textrm{if $x \in K_1-\{a\}$},\\
x & \textrm{if $x \in H$.}
\end{cases}
\]
An easy argument shows that $r$ is a one-step deformation of $G$ onto $H$, and since $H$ has trivial reduced homology in both the cubical and path case, the result follows from \ref{equiv-thm}.
\end{proof}

\begin{Corollary} For all $s,t >0$, let $K_{s,t}$ denote a complete bipartite graph with $s+t$ vertices. Then $\Hom_n^\Path(K_{s,t}) \cong \Hom_n^\Cube(K_{s,t}) \cong (0)$ for $n>0$. \hfill \qed
\end{Corollary}

\begin{Corollary} \label{cor-join}Let $K_1$ and $K_2$ be graphs with disjoint vertex sets. Consider the join graph $G = K_1 \ast K_2,$  where
$V(G) = V(K_1) \cup V(K_2)$ and $E(G)$ consists of $E(K_1)$ and $E(K_2)$ together
with all edges $\{p,q\}$ connecting a vertex $p\in V(K_1)$ with a vertex  $q \in V(K_2)$. Then $\Hom_n^\Path(K_1\ast K_2) \cong \Hom_n^\Cube(K_1\ast K_2) \cong (0)$ for $n>0$. \hfill \qed
\end{Corollary}

The last corollary includes two elementary but important examples, the {\em cone} $G*\{p\}$ of $G$ over $p$, and the {\em suspension} $G*\{p,q\}$ of $G$ over a pair of non-adjacent vertices $p$ and $q.$ \Ref{cor-join} shows that the reduced cubical and path homologies in both cases are trivial.

\begin{Definition}
The {\em disjoint sum} $K_1\oplus K_2$ of graphs $K_1$ and $K_2$ is the graph with vertex set $V(K_1\oplus K_2) = V(K_1)\cup V(K_2)$ and edge set $E(K_1\oplus K_2) = E(K_1)\cup E(K_2)$.
\end{Definition}

\begin{Theorem}\label{disjointsum}
For any graphs $K_1$ and $K_2$, 
$\Hom_n^\Cube(K_1\oplus K_2) \cong 
\Hom_n^\Cube(K_1)\oplus \Hom_n^\Cube(K_2)$ and 
$\Hom_n^\Path(K_1\oplus K_2) \cong 
\Hom_n^\Path(K_1)\oplus \Hom_n^\Path(K_2)$ for all $n\ge 0$.
\end{Theorem}
\begin{proof}
The proof is elementary in both cases.
\end{proof}

\begin{Definition}\label{def-chordal}
A graph $G$ is {\em chordal} if every cycle of length greater than three contains a chord. Equivalently, $G$ is chordal if and only if there exists an ordering of its vertices
$v_1,\dots,v_m$ such that for each $j>1$, the set of vertices $v_k$ adjacent to $v_j$ with $k<j$ form a clique (possibly empty). 
\end{Definition}

\begin{Theorem}
If $G$ is a chordal graph, then $\Hom_n^\Path(G) \cong \Hom_n^\Cube(G) \cong (0)$ for $n>0$.
\end{Theorem}

\begin{proof}
Suppose that $n>0$ and $v_1,\dots,v_m$ is an ordering of $V(G)$ satisfying the condition of \Ref{def-chordal}. For $j \in [m]$, let $G^{(j)}$ denote the  induced subgraph of $G$ whose vertex set is 
$\{v_1,\dots,v_j\}$. Proceeding by induction, suppose that
$\Hom_n^\Path(G^{(j)}) \cong \Hom_n^\Cube(G^{(j)}) \cong (0)$. If $v_{j+1}$ has no neighbors in $G^{(j)}$, it follows from 
\Ref{disjointsum} that 
$\Hom_n^\Path(G^{(j+1)}) \cong \Hom_n^\Cube(G^{(j+1)}) \cong (0)$. Otherwise, suppose that $v_{j+1}$ has neighbors in $G^{(j)}$ and let $v_k$ with
$k<j+1$ be one of them. It is easy to check that the map from 
$G^{(j+1)}$ to $G^{(j)}$ defined by sending $ v_{j+1}$ to $v_k$ and fixing the remaining elements of $G^{(j)}$ is a $1$-step deformation retraction. Hence $G^{(j+1)}$ has trivial homology, by \Ref{equiv-thm}.
\end{proof}

The next theorem shows how the homology theories $\Hom^\Cube_\bullet$ and $\Hom_\bullet^\Path$ behave with respect to three well-known types of graph products.  
One of these, the {\em box product} $G\,\Box\, H$ has already been defined in \Ref{boxdef}. The next definition introduces two more.
For a more complete treatment of these constructions, see \cite{Hammack11}. 

\begin{Definition}
Suppose that $G$ and $K$ are graphs. Define the {\em strong product} $G\boxtimes K$ and the {\em lexicographic product} $G[K]$ as graphs whose vertex set is the Cartesian product set $V(G)\times V(K)$, and whose edges are 
pairs $\{(g_1,k_1), (g_2,k_2)\}$
defined by the following rules:
\begin{enumerate}
\item
$(g_1,k_1) \sim_\boxtimes (g_2,k_2) \textrm{ iff }
((g_1=g_2)\wedge (k_1 \sim k_2))\vee ((g_1 \sim g_2) \wedge (k_1=k_2))\vee ((g_1 \sim g_2) \wedge (k_1\sim k_2))$
\item
$(g_1,k_1) \sim_\textrm{lex}
(g_2,k_2)\textrm{ iff }
((g_1\sim g_2))\vee ( (g_1 = g_2) \wedge (k_1\sim k_2))$.
\end{enumerate}

\end{Definition}

\begin{Theorem}\label{retract-products}
Suppose that $G$ and $K$ are graphs, and $H$ is an induced subgraph of $K$. Suppose that $r: V(K)\to V(H)$ is a retraction of $K$ onto $H$. Then for all $n\ge 0$,
\begin{enumerate}
\item
$\Hom_n^\Cube(G\,\Box\, K) \cong \Hom_n^\Cube(G\,\Box\, H)$ \textrm{ and }
$\Hom_n^\Path(G\,\Box \, K) \cong \Hom_n^\Path(G\, \Box\, H)$,
\item
$\Hom_n^\Cube(G\boxtimes K) \cong \Hom_n^\Cube(G \boxtimes H)$ \textrm{ and }
$\Hom_n^\Path(G\boxtimes K) \cong \Hom_n^\Path(G \boxtimes H)$,
\item
$\Hom_n^\Cube(G[K]) \cong \Hom_n^\Cube(G[H])$ \textrm{ and }
$\Hom_n^\Path(G[K]) \cong \Hom_n^\Path(G[H])$.
\end{enumerate}
\end{Theorem}

\begin{proof}
It suffices to prove that the map 
\[
(g,k) \longmapsto (g,r(k)) 
\]
defines a retraction from $G\,\Box\, H$ to $G\, \Box\, K$, from 
$G\boxtimes H$ to $G \boxtimes K$, and from $G[K]$ to $G[H]$. The arguments in 
each case are straightforward.
\end{proof}

\section{A map between the chain complexes $\CC^\Cube(G)$ and $\CC^\Path(G)$}
\label{sec-map}

In this section we establish a map between the chain complexes $\CC^\Cube(G)$ and $\CC^\Path(G)$.  
Consider a singular $n$-cube $\sigma : Q_n \rightarrow G$. 
In order to define a map from $\CC^\Cube_n(G)$ to $\CC_n^\Path(G)$, we first 
associate to any permutation
$\tau \in S_n$ a path 
$p_\tau$ from $(0,\ldots,0) \in V(Q_n)$ to 
$(1,\ldots,1) \in V(Q_n)$.
The path $p_\tau$ is defined as the path of length $n-1$ which in its 
$i$\textsuperscript{th} step
flips the $\tau(i)$\textsuperscript{th} coordinate from $0$ to $1$.
We write $p_\tau(i)$ for the $i$\textsuperscript{th} vertex in the path
$p_\tau$, $0 \leq i \leq n$.

To $\sigma \in \CC_n^\Cube(G)$, we assign the element

$$\psi(\sigma) := \sum_{\tau \in S_n} \sgn(\tau) \, \sigma \circ p_\tau$$

\noindent of $\CC_n^\Path(G)$, where $\sigma \circ p_\tau$ denotes the path in $G$
whose $i$\textsuperscript{th} vertex is $\sigma(p_\tau(i))$.
Note that we can have $\sigma \circ p_\tau = 0$ as there may be adjacent 
vertices in $Q_n$ that are sent by $\sigma$ to the same vertex of $G$. 

\begin{Lemma}
  \label{map}
  Let $\sigma \in \CC_n^\Cube(G)$. Then
  { \setstretch{1.5}
  \begin{itemize}
    \item[(i)] 
       $\partial_{n}^\Path \psi(\sigma) \in \LL_{n-1}^\Path(G).$ In 
       particular, $\psi(\sigma) \in \CC_n^\Path(G)$. 
    \item[(ii)] 
       $\partial_{n}^\Path \psi(\sigma) = \psi(\partial_{n}^\Cube (\sigma)).$
  \end{itemize}}
\end{Lemma}
\begin{proof}
   For part (i) we have
  \begin{eqnarray} 
    \label{eq:yaudiff}
      \partial_n^\Path (\sigma \circ p_\tau) & = & \sum_{i=0}^n (-1)^i \cdot 
         \sigma (p_\tau(0)) \cdots \widehat{\sigma( p_\tau(i))} \cdots 
         \sigma(p_\tau(n)). 
    \end{eqnarray} 
    Let $1 \leq \ell \leq n-1$. If $\tau'$ is constructed from $\tau$ by interchanging 
    $\tau(\ell)$ and $\tau(\ell+1)$ then $p_\tau(i) = p_{\tau'}(i)$ for $i \neq \ell$.
    In particular, the $\ell$\textsuperscript{th} summands of
    \eqref{eq:yaudiff} for $p_\tau$ and $p_{\tau'}$ coincide. In addition, we 
     have $\sgn(\tau) = -\sgn(\tau')$. 
     This shows that 
     \begin{eqnarray}
       \label{eq:yaudiff2}
       \partial_n^\Path \psi(\sigma) & = & \sum_{\tau \in S_n} \sgn(\tau) \, \partial_n^\Path (\sigma \circ p_\tau) \\ 
                  \nonumber  & = & \sum_{\tau \in S_n} \sgn(\tau) \, \Big( (\sigma(p_\tau (1)) ,\ldots, \sigma( p_\tau (n))) + \\
                  \nonumber  &   & \qquad\qquad\qquad\qquad(-1)^n (\sigma(p_\tau(0)) ,\ldots, \sigma(p_\tau(n-1))) \Big). 
     \end{eqnarray}
    Since both $(\sigma(p_\tau (1)), \ldots ,\sigma( p_\tau (n)))$ and
    $(\sigma(p_\tau (0)) \cdots \sigma( p_\tau (n-1)))$ are paths it follows that
    $\partial_{n}^\Path \psi(\sigma) \in \LL_{n-1}^\Path(G)$. As a consequence
    $\psi(\sigma) \in \CC_n^\Path(G)$, and we have proved (i).
    
   For part (ii), suppose that $\tau \in S_{n}$. Define 
    \begin{itemize}
       \item $\tau^- \in S_{n-1}$ to be the permutation
          where $\tau^- (j) = \tau(j)$ if $\tau(j) < \tau(n)$ and $\tau(j)-1$ if 
          $\tau(j) > \tau(n)$. 
       \item $\tau^+ \in S_{n-1}$ to be the permutation
          where $\tau^+ (j) = \tau(j+1)$ if $\tau(j+1) < \tau(1)$ and $\tau(j+1)-1$ if 
          $\tau(j+1) > \tau(1).$ 
   \end{itemize} 
   Now for $j = \tau(n)$ and $j' = \tau(1)$ we have  
   \begin{eqnarray*}
      (\sigma(p_\tau(0)),\ldots, \sigma(p_\tau(n-1))) & = & (\sigma (p_{\tau^-} (0))
          \ldots \sigma (p_{\tau^-} (n-1))), \\  
      (\sigma(p_\tau(1)),\ldots, \sigma(p_\tau(n))) & = & (\sigma (p_{\tau^+} (0),
          \ldots , \sigma (p_{\tau^+}(n-1))).
   \end{eqnarray*}
   Now 
   \begin{eqnarray*}
      \psi(f_i^- \sigma) & = & \sum_{\tau' \in S_{n-1}} \sgn(\tau') \, (f_i^- \sigma) \circ p_{\tau'} \\ 
      \psi(f_i^+ \sigma) & = & \sum_{\tau' \in S_{n-1}} \sgn(\tau') \, (f_i^+ \sigma) \circ p_{\tau'} \\ 
   \end{eqnarray*}

   Since for $\tau \in S_n$ we have that 
   $\tau$ is determined by $\tau^-$ and $\tau(n)$, as well as by $\tau^+$ 
   and $\tau(n)$ it follows that:
   \begin{eqnarray*}
      \psi(f_i^- \sigma) & = & \sum_{{\tau \in S_{n}},{\tau(n) = i}} \sgn(\tau^-) \, (f_i^- \sigma) \circ p_{\tau^-} \\ 
      \psi(f_i^+ \sigma) & = & \sum_{{\tau \in S_{n}}, {\tau(1) = i}} \sgn(\tau^+) \, (f_i^+ \sigma) \circ p_{\tau^+} 
   \end{eqnarray*}
   Since $\tau(1)=i$ contributes $i-1$ inversions and $\tau(n) = i$ in the last position $n-i$ inversions to $\tau$ we obtain 
   \begin{eqnarray}
      \label{eq:faceim1}
        \psi(f_i^- \sigma) & = & (-1)^{n-i} \sum_{{\tau \in S_{n}},{\tau(n) = i}} \sgn(\tau) \, (f_i^- \sigma) \circ p_{\tau^-} \\ 
      \label{eq:faceim2}
        \psi(f_i^+ \sigma) & = & (-1)^{i-1} \sum_{{\tau \in S_{n}} ,{\tau(1) = i}} \sgn(\tau) \, (f_i^+ \sigma) \circ p_{\tau^+} 
   \end{eqnarray}
   We have 
   \begin{eqnarray*}
      \partial_n^\Cube \sigma & = & \sum_{i=1}^n (-1)^i (f_i^- \sigma - f_i^+ \sigma).
   \end{eqnarray*}
   Thus by \eqref{eq:faceim1}, \eqref{eq:faceim2} and \eqref{eq:yaudiff2} we get  
   \begin{eqnarray*}
      \psi(\partial_n^\Cube \sigma) & = & \sum_{i=1}^n (-1)^i \, (\psi(f_i^- \sigma) - \psi(f_i^+ \sigma)) \\
                                    & = & \sum_{i=1}^n (-1)^i \Big( (-1)^{n-i} \sum_{{\tau \in S_{n}}, {\tau(\ell) = i}} \sgn(\tau) \, (f_i^- \sigma) \circ p_{\tau^-} \, - \\
                                &   & (-1)^{i-1} \sum_{{\tau \in S_{n}} , {\tau(1) = i}} \sgn(\tau) \, (f_i^+ \sigma) \circ p_{\tau^+} \Big) \\
                                & = &  \sum_{\tau \in S_n} \sgn(\tau) 
                 \Big( (-1)^n (\sigma(p_\tau(0)), \ldots,\sigma(p_\tau(n-1))) + \\
                                &   & 
 (\sigma (p_\tau (1),\ldots, \sigma(p_\tau (n)) \Big) \\
                                & = & \partial_n^\Path \psi(\sigma) .
   \end{eqnarray*} 

  \end{proof}

  Next we show that degenerate cubes are mapped to zero under $\psi$. 

  \begin{Lemma}
    \label{degen}
    For any degenerate singular $n$-cube $\sigma\in \LL_n^\Cube(G)$ we have 
    $\psi(\sigma) = 0$.
    In particular, $\psi$ induces a map $\psi : \CC_n^\Cube(G) \rightarrow \CC_n^\Path(G)$.
  \end{Lemma}
  \begin{proof}
     If $\sigma \in \LL_n^\Cube(G)$ is degenerate, then there is an $i\in[n]$ such that
     $f_i^- \sigma = f_i^+ \sigma$.
     Now each path $\sigma \circ p_\tau$ in $\psi(\sigma)$ at some point passes from 
     the facet $f_i^{-} \sigma$ to the
     facet $f_i^{+} \sigma$ of the singular $n$-cube $\sigma$. 
     Since  $f_i^- \sigma = f_i^+ \sigma$,
     it follows that there are two consecutive vertices in the path 
     $\sigma \circ p_\tau$ that are 
     identical. This shows
     $\sigma \circ p_\tau = 0$ in $\LL_n^\Path(G)$.   
  \end{proof}
  
  For small $n$ we have good control over $\psi$.
  
  \begin{Lemma}
    \label{lem:smalln}
    The map $\psi : \CC_n^\Cube(G)
    \rightarrow \CC_n^\Path(G)$ is an isomorphism for
    $n \leq 1$ and surjective for $n =2$. 
  \end{Lemma}
  \begin{proof}
     For $n = 0$ both $\CC_0^\Cube(G)$ and $\CC_0^\Path(G)$ 
     are freely generated by the vertices of $G$. Since $\psi$ maps a vertex considered a the image of a $0$-cube to 
     the path consisting of that vertex, it is clearly an
     isomorphism.
     
     Consider $n = 1$. A non-degenerate $1$-cube is represented by an edge with
     a direction chosen. Hence $\CC_1^\Cube(G)$ has a basis given by
     pairs of vertices $(v_1,v_2)$ where $\{v_1,v_2\}$ is an edge in the graph.
     But this is also a basis for $\CC_1^\Path(G)$. 
     Thus $\CC_1^\Cube(G)$ and $\CC_1^\Path(G)$ are isomorphic and $\psi$ is
     an isomorphism.
     
     Finally we turn to $n = 2$. 
     Let $\sigma \in \CC_2^\Path(G)$. Then, by the proof of 
     \cite[Proposition 4.2]{GLMY1}, $\sigma$ is a linear
     combination of paths $(v,v',v)$ and $(v,v',v'')$ where 
     $\{v,v''\}$ is an
     edge in $G$, and of chains of the form $(v,w,v') -(v,w',v')$ 
     where $w \neq w'$ and $\{v,v'\}$ is not an edge in $G$. 
     We show that all those chains are in the image of $\psi$.
     Indeed:
     \begin{center}
       \begin{tabular}{lcr}
          \begin{minipage}[t]{0.4\textwidth} 
            $$\psi \Big(
              \begin{array}{ccc} 
                 v & \rule[3pt]{1cm}{1pt} & v \\
                 \vrule & & \vrule \\
                 v & \rule[3pt]{1cm}{1pt}& v' 
             \end{array} \Big) = (v,v',v), $$
          \end{minipage}
          &  & 
          \begin{minipage}[t]{0.3\textwidth}
            $$\psi \Big(
            \begin{array}{ccc} 
                v & \rule[3pt]{1cm}{1pt} & v'' \\
                \vrule & & \vrule \\
                v & \rule[3pt]{1cm}{1pt}& v' 
            \end{array} \Big) = (v,v',v''),$$
          \end{minipage} \\ 
          \begin{minipage}[t]{0.4\textwidth}
            $$\psi \Big(
              \begin{array}{ccc} 
                 w' & \rule[3pt]{1cm}{1pt} & v' \\
                 \vrule & & \vrule \\
                 v & \rule[3pt]{1cm}{1pt}& w 
              \end{array} \Big) = (v,w,v')-(v,w',v').$$
           \end{minipage}
     & ~~~~ &
    \end{tabular}
    \end{center}
  \end{proof}
  
  Now we are in position to state and prove the consequence of
  the preceding lemmas on the relation of the two homology theories. 
  
  \begin{Theorem}
    \label{pr:yaubarcelo}
    For any $n \geq 0$, the map $\psi$ induces a homomorphism 
    $\psi_* : \Hom_n^\Cube (G) \rightarrow \Hom_n^\Path(G)$.
    For $n \leq 1$ the map $\psi_*$ is an isomorphism and
    for $n =2$ it is surjective.
  \end{Theorem}
   \begin{proof}
     The fact that $\psi$ induces a homomorphism 
     $\psi_* : \Hom_n^\Cube (G) \rightarrow \Hom_n^\Path(G)$ is
     immediate from \ref{map} and \ref{degen}. 
     Next we consider $n=0,1,2$. 
     
     By \ref{lem:smalln} we can identify 
     $\CC_1^\Cube(G)$ and $\CC_1^\Path(G)$
     by considering a non-degenerate $1$-cube $\sigma$ as the edge
     $\psi(\sigma) = (\sigma(0),\sigma(1))$.
     After this identification, the differentials $\partial_1^\Cube$ and 
     $\partial_1^\Path$ are easily seen to be identical.
     In particular, we have $\Ker\,\partial_1^\Path = \Ker\, \partial_1^\Cube$. 
     Thus we need to show that $\Im\, \partial_2^\Cube = \Im\, \partial_2^\Path$.
     By \ref{lem:smalln}, $\psi$ is an isomorphism in dimension $1$, and
     since $\psi \circ \partial_2^\Cube = \partial_2^\Path \circ \psi$,  it follows that
     $\Im\,\partial_2^\Cube \subseteq \Im\, \partial_2^\Path$. 
     Conversely, let $\sigma \in \Im\, \partial_2^\Path$. 
     Then there is $\sigma' \in \CC_2^\Path(G)$ such that
     $\partial_2^\Path(\sigma') = \sigma$. By \ref{lem:smalln}
     there is $\sigma'' \in \CC_2^\Cube(G)$ such that
     $\psi(\sigma'') = \sigma'$. Then, again 
     by $\psi \circ \partial_2^\Cube = \partial_2^\Path \circ \psi$,  it follows that $\sigma = \partial_2^\Cube$ and 
     $\Im\, \partial_2^\Path \subseteq \Im\, \partial_2^\Cube$. 
     This implies that $\psi_* : \Hom_1^\Cube(G) 
     \rightarrow \Hom_1^\Path(G)$ is an isomorphism.
     
     Now consider homological dimension $2$. 
     Let $\sigma + \Im\, \partial_2^\Path$ be an element of
     $\Hom_2^\Path(G)$. By \ref{lem:smalln} we
     know that $\phi$ is surjective in dimension $2$. Hence 
     $\phi^{-1}(\sigma)$ is non-empty. For
     $\sigma' \in \phi^{-1}(\sigma)$ we have
\begin{eqnarray*}
       \psi (\partial_2^\Cube(\sigma')) & = & \partial_2^\Path (\psi(\sigma')) \\
                                      & = & \partial_2^\Path (\sigma) \\
                                      & = & 0. 
    \end{eqnarray*}
    From \ref{lem:smalln} we know that $\psi$ is bijective 
    in dimension $1$. From that we deduce
    $\sigma' \in \Ker\, \partial_2^\Cube$. Thus
    $\sigma' + \Im\,\partial_3^\Cube(G) \in \Hom_2^\Cube(G)$
    and $\psi_*(\sigma' + \Im\, \partial_2^\Cube) = \sigma+
    \Im\, \partial_2^\Path$. Thus $\psi_*$ is surjective.
   \end{proof}

   In \cite[Theorem 1.2]{BCW} it is shown that that the abelianization of
   the discrete fundamental group $A_1(G)$ (see for example \cite{BBLL} for definitions) of a graph is isomorphic to 
   $\Hom_1^\Cube(G)$. In \cite[Theorem 4.23]{GLMY2} it is shown that
   the abelianization of the discrete fundamental group of a graph
   is isomorphic to $\Hom_1^\Path(G)$. Indeed their result is more general
   and captures all directed graphs. Now \ref{pr:yaubarcelo} 
   can be used to deduce either result from the other. 

   \begin{Corollary}
     Let $G = (V,E)$ be a graph then there are isomorphisms:
     $$\Hom_1^\Cube(G) \cong A_1(G)/[A_1(G),A_1(G)] 
                       \cong \Hom_1^\Path(G).$$ 
   \end{Corollary}

   \ref{pr:yaubarcelo} also raises the question if 
   $\Hom_n^\Cube(G) \cong \Hom_n^\Path(G)$ for all $n$. While we have shown
   in Section 4 that this is true for many graphs $G$, it is
   false in general.

\section{Example: $\Hom^\Cube_\bullet$ and $\Hom^\Path_\bullet$ Are Not Always Isomorphic}
\label{sec-counter}

Results in the previous two sections might suggest conjecturing that
$\Hom_\bullet^\Cube(G)$ and $\Hom_\bullet^\Path(G)$ are the same for all graphs $G$.
In this section we construct an example showing that this is not always the case.  From \ref{pr:yaubarcelo} we know that
$\psi_* :
\Hom_2^\Cube(G) \rightarrow \Hom_2^\Path(G)$ is surjective. 
The next theorem shows that it is not always injective.
 
 \medskip
\noindent
\begin{Theorem}
  \label{thm:counterexample}
  Let $G$ be the following graph:
  
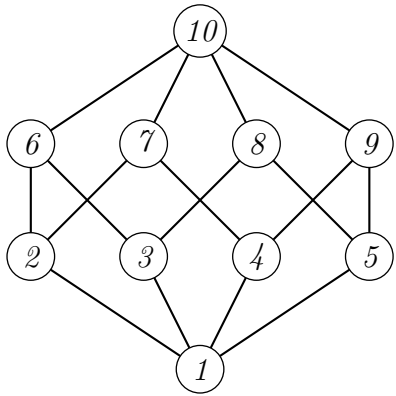
\begin{center}
\begin{tikzpicture}[scale=1.5]
   \Vertex[x=1.5,y=0]{1}
   \Vertex[x=0 ,y=1]{2}
   \Vertex[x=1,y=1]{3}
    \Vertex[x=2,y=1]{4}
     \Vertex[x=3,y=1]{5}
      \Vertex[x=0,y=2]{6}
      \Vertex[x=1,y=2]{7}
      \Vertex[x=2,y=2]{8}
          \Vertex[x=3,y=2]{9}
                 \Vertex[x=1.5,y=3]{10}
                                \Edge(1)(2)
    \Edge(1)(3)  \Edge(1)(4)  \Edge(1)(5)  
    \Edge(2)(6) \Edge(2)(7) \Edge(3)(8)\Edge(3)(6)
    \Edge(4)(7) \Edge(4)(9)\Edge(5)(8)\Edge(5)(9)
    \Edge(6)(10)\Edge(7)(10)\Edge(8)(10)\Edge(9)(10)
    \end{tikzpicture}
\end{center}
Then $\Hom^\Path_2(G) \cong (0)$ and $\Hom^\Cube_2(G) \not\cong (0)$.
\end{Theorem}

\begin{proof} 
We first prove that $\Hom^\Path_2(G) \cong (0)$, by showing explicitly that every $2$-cycle is a boundary. Suppose that 
$\theta\in \CC^\Path_2(G)$ is a $2$-cycle, i.e. $\partial_2 \theta = 0$.  We claim that $\theta$ is a linear combination of cycles of one of the following
three types:
\begin{itemize}
\item[(1)] $(b,a,b) - (a,b,a)$, where $\{a,b\}$ is an edge of $G$,
\item[(2)] $((a,b,c)-(a,d,c)) + ((d,c,b)-(d,a,b)) + (d,a,d) - (c,b,c)$, where $a,b,c,d$, are consecutive vertices of a quadrilateral,
\item[(3)] the cycle
\begin{gather*}
((6,2,1)-(6,3,1))+ ((8,3,1)-(8,5,1)) + ((9,5,1)-(9,4,1))\qquad\qquad \\\ \qquad\qquad+ ((7,4,1)-(7,2,1)) 
 + ((10,6 ,3)- (10,8, 3))  + ((10,8, 5) - (10,9, 5)) \\ \qquad \qquad\qquad\qquad +((10,9,4) - (10,7 ,4)) +((10, 7, 2)-(10, 6, 2)).
\end{gather*}
\end{itemize}
One can visualize $G$ as the $1$-skeleton of a polytope with eight quadrilateral facets. The cycle (3) is obtained by giving each quadrilateral an outward orientation,  then assigning signs to paths around each quadrilateral using a right-hand rule.

Each of (1), (2), and (3) is easily seen to be a boundary (and hence a cycle):
\begin{gather*}
(b,a,b) - (a,b,a) = \partial_3 \big( (a,b,a,b) \big) \\[+1ex]
((a,b,c)-(a,d,c)) + ((d,c,b)-(d,a,b)) + (d,a,d) - (c,b,c) \qquad\qquad\qquad\qquad \\ \qquad\qquad\qquad\qquad\qquad= \partial_3 \big(
(d,a,b,c) - (d,a,d,c) - (d,c,b,c) \big), 
\end{gather*}
and (3) is equal to
\begin{gather*}
 \partial_3 \big((10,6,2,1) - (10 ,6 ,3 ,1) + (10 , 7 ,4 ,1)- (10 , 7 ,2, 1) \qquad\qquad \\ \qquad\qquad+ (10 ,8, 3, 1)- (10,8, 5, 1)+ (10,9, 5, 1) - (10,9 ,4 ,1)\big).
\end{gather*}
To complete the argument, we must show that every $2$-cycle can be expressed as a linear combination of cycles of type (1), (2), and (3).  If 
 $\theta$ is a $2$-cycle, let $(a,b,c)$ be 
 the lexicographically first term in $\theta$ with the following properties: (i) it is injective, i.e. not of the 
form $(a,b,a)$, and (ii) it is not monotone decreasing, i.e. not satisfying $a>b>c$. Let us call an injective term 
$(a,b,c)$ ``bad" if it satisfies property (ii), and ``good" otherwise (i.e. if it is decreasing).

Necessarily, such an $(a,b,c)$ must be paired in $\theta$ with another 
opposite-signed term $(a,d,c)$ where $\{d,c\}$ and $\{a,d\}$ form the edges opposite to 
$\{a,b\}$ and $\{b,c\}$ in a quadrilateral of $G$ (otherwise the term $(a,c)$ in $\partial_2((a,b,c))$ does not cancel). 
Since $(a,b,c)$ is lexicographically first in $\theta$, we must have $b<d$.  
We claim further that $a<d$. If $a<b$ this is immediate; if $a>b<c$, it is easy to check that $a>d$ does not hold 
in any of the eight quadrilaterals in $G$. 

If $\tau$ is the canonical cycle of type (2) above, then since $b<d$ and $a<d$, every injective term in $\tau$ follows $(a,b,c)$ in lexicographic order. Hence
we can use $\tau$ to eliminate $(a,b,c)$ from $\theta$, and by repeating the process eventually arrive at a cycle $\theta^*$ in which every injective term is ``good" , i.e., of the form  $(a,b,c)$ with $a>b>c$.

In the boundary $\partial_2 \theta^*$, all $1$-chains arising from good injective terms must be of the form 
$(x,y)$ with $x>y$.  Hence if $\theta^*$ contains a non-injective term $(a,b,a)$, it must also contain a corresponding term of the form
$(b,a,b)$, and we can cancel them both out by subtracting a cycle of type (2). Eventually we arrive at a cycle
$\theta^{**}$ in which every term is injective and ``good''. 

As noted above, all injective terms must appear in pairs $(a,b,c), (a,d,c)$ arising from one of the eight quadrilaterals in $G$, with opposite-signed coeffients of equal magnitude. 
We may thus regard the equation $\partial_2 \theta^{**} = 0$ as a homogeneous linear system with eight unknowns (corresponding to the ``good'' quadrilateral boundary pairs appearing in (3)) and 16
equations corresponding to the coefficients of the $1$-chains 
\begin{gather*}(2,1), (3,1), (4,1), (5,1), (6,2), (6,3), (7,2), (7,4), (8,3), (8,5), \\ (9,4), (9,5), (10,6),(10,7),(10,8),(10,9).
\end{gather*}
This system has a matrix
\[
\left(
\begin{array}{cccccccc}
 1 & 0 & 0 & -1 & 0 & 0 & 0 & 0 \\
 -1 & 1 & 0 & 0 & 0 & 0 & 0 & 0 \\
 0 & 0 & -1 & 1 & 0 & 0 & 0 & 0 \\
 0 & -1 & 1 & 0 & 0 & 0 & 0 & 0 \\
 1 & 0 & 0 & 0 & 0 & 0 & 0 & -1 \\
 -1 & 0 & 0 & 0 & 1 & 0 & 0 & 0 \\
 0 & 0 & 0 & -1 & 0 & 0 & 0 & 1 \\
 0 & 0 & 0 & 1 & 0 & 0 & -1 & 0 \\
 0 & 1 & 0 & 0 & -1 & 0 & 0 & 0 \\
 0 & -1 & 0 & 0 & 0 & 1 & 0 & 0 \\
 0 & 0 & -1 & 0 & 0 & 0 & 1 & 0 \\
 0 & 0 & 1 & 0 & 0 & -1 & 0 & 0 \\
 0 & 0 & 0 & 0 & 1 & 0 & 0 & -1 \\
 0 & 0 & 0 & 0 & 0 & 0 & -1 & 1 \\
 0 & 0 & 0 & 0 & -1 & 1 & 0 & 0 \\
 0 & 0 & 0 & 0 & 0 & -1 & 1 & 0 \\
\end{array}
\right),
\]
which is easily seen to have rank $7$, and hence all solutions to $\partial \theta^{**} = 0$ are constant multiples of (3).
Since cycles of type (1), (2), and (3) are all boundaries, this completes the proof that every $2$-cycle is a boundary, and hence
$\Hom^\Path_2(G) \cong (0)$.
\end{proof}

Next we turn to proving that $\Hom^\Cube_2(G) \not\cong (0)$, which will be done by finding an explicit $2$-cycle 
$\theta\in \CC_2^\Cube(G)$ that is not a boundary.  Define  
\begin{equation}\label{theta}
\begin{gathered}
\theta =(1,2,3,6) - (1,2,4,7) + (1,3,5,8) - (1,4,5,9) \qquad\qquad \\
\qquad\qquad- (2,6,7,10) + (3,6,8,10) - (4,7,9,10) + (5,8,9,10).
\end{gathered}
\end{equation}
Recall that each sequence denotes a labeling of the canonical $2$-cube by vertices in $G$, proceeding recursively by dimension. For example, $(1,2,3,6)$ represents the labeling
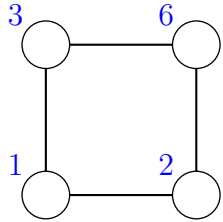
\begin{center}
\begin{tikzpicture}[scale=2]
   \Vertex[x=0 ,y=0, L=$$]{1}
   \Vertex[x=1 ,y=0, L=$$]{2}
   \Vertex[x=0,y=1, L=$$]{3}
    \Vertex[x=1,y=1, L=$$]{6}
   \Edge(1)(2)
    \Edge(1)(3)  \Edge(2)(6)  \Edge(3)(6)  
     \node  [color=blue] at ([shift={(-.2,.2)}]1) {1};
     \node  [color=blue] at ([shift={(-.2,.2)}]2) {2};
     \node  [color=blue] at ([shift={(-.2,.2)}]3) {3};
     \node  [color=blue] at ([shift={(-.2,.2)}]6) {6};
    \end{tikzpicture}
\end{center}
One can interpret $\theta$ as the result of wrapping quadrilaterals around a $3$-polytope in an orientation-preserving way. When labeled properly by their vertex names, each face can be viewed as graph homomorphisms from a $2$-cube into $G$. 

It is easy to check that $\theta$ is a cycle, i.e., $\partial_2 \theta= 0$.
We will show that $\theta$ is not a boundary, by constructing a linear invariant $\Psi$ 
on $\CC^\Cube_2(G)$ that is zero on every $2$-boundary but is nonzero on $\theta$.  
Fix a quadrilateral $Q_0$ in $G$, say $Q_0=(1,2,3,6)$ with vertices listed in increasing order. We say that a $2$-cell
$F = (a,b,c,d)$
 (that is, a $G$-labeled $2$-cube, with labels in the standard reading order) is {\em supported by $Q_0$} if its labels agree with those of $Q_0$ in some order. Since $F$ is a graph map, the permutation $\sigma$ mapping $(1,2,3,6)$ onto $(a,b,c,d)$ is an element of the dihedral group $D_4$. Define the {\em weight} $w(F)$ of $F$ to be $\chi(\sigma)$, where $\chi$ is the reflection character of $D_4$. In other words, $\chi(\sigma) = \pm 1$ according to whether $\sigma$ is a reflection. If $X$ is any $2$-chain, define $\Psi(X)$ to
 be the sum over $X$ of the coefficient of each $2$-cell times the weight $w(F)$ of that cell.  In this computation, $w(F)=0$ if
 $F$ is not supported by $Q_0$.
 
 The following rules define $\Psi$ explicitly on the eight $2$-cells supported by $Q_0$:
 \[
 \Psi: 
 \begin{cases}
 (1,2,3,6) & \longmapsto +1 \\
 (2,6,1,3) & \longmapsto +1 \\
 (3,1,6,2) & \longmapsto +1 \\
 (6,3,2,1) & \longmapsto +1\\
 (1,3,2,6) & \longmapsto -1 \\
 (2,1,6,3) & \longmapsto -1 \\
 (3,6,1,2) & \longmapsto -1 \\
 (6,2,3,1) & \longmapsto -1
 \end{cases}
 \]
 As an illustration, note that $\Psi(\theta) = 1$, since $(1,2,3,6)$ is the only $2$-cell appearing in $\theta$ 
 supported by $Q_0$.
 
 As another illustration, consider the $3$-cell
 \[
 Y = (3,6,1,3,1,2,3,6),
 \]
 which has $Q_0 = (1,2,3,6)$ as its top face.  Its boundary is
 \[
 \partial_3 Y =  (1,2,3,6) - (1,3,3,6) - (3,1,1,3)+(3,6,1,2) - (3,6,1,3) + (6,3,2,6).
 \]
 In this case, two terms are supported by $Q_0$ but they appear with opposite signs when $\Psi$ is applied, and we get 
 $\Psi (\partial_3 Y) = 0$.  This turns out to be a general phenomenon:
 
\begin{Lemma} For any non-degenerate $3$-cell $Y$, we have $\Psi(\partial_3 Y) = 0$. Consequently, $\Psi(X) = 0 $ for any $2$-boundary $X$.
\end{Lemma}

\begin{proof} 
Suppose that $Y$ is a non-degenerate $3$-cell.  We claim first that the number of $2$-faces supported by $Q_0= (1,2,3,6)$ is equal to $2$ or $4$. 
To prove this, we systematically eliminate the other cases. It is easy to see that the number of $Q_0$-supported $2$-faces cannot be $3, 5$, or $6$, since if two such faces are adjacent dihedrally, the two faces adjacent to both of of those faces have a repeated label, and hence cannot be supported by $Q_0$. 

It remains to show that the number of $Q_0$-supported faces cannot be $1$.  Suppose that $Y$ contains only one $Q_0$-supported face, e.g. as indicated in the following picture where the bottom four vertices are labeled by $1,2,3$ and $6$, and the others are labeled generically by $A,B,C$ and $D$.

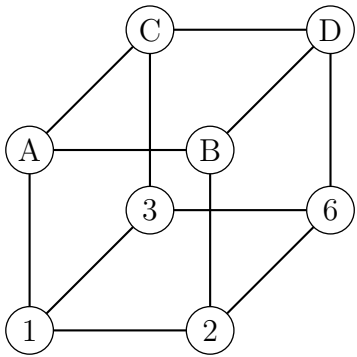
\begin{center}
\begin{tikzpicture}[scale=.8]
   \Vertex[x=0 ,y=0]{1}
   \Vertex[x=3 ,y=0]{2}
   \Vertex[x=2,y=2]{3}
    \Vertex[x=5,y=2]{6}
      \Vertex[x=0,y=3]{A}
        \Vertex[x=3,y=3]{B}
          \Vertex[x=2,y=5]{C}
            \Vertex[x=5,y=5]{D}
   \Edge(1)(2)
    \Edge(1)(3)  \Edge(2)(6)  \Edge(3)(6)  
    \Edge(A)(B) \Edge(C)(D) \Edge(A)(C)
    \Edge(B)(D)
    \Edge(1)(A) \Edge(2)(B) \Edge(3)(C) \Edge(6)(D)
    \end{tikzpicture}
\end{center}
The label $B$ can be either $1,2, 6$, or $7$, since these are the vertices adjacent to $2$ in $G$. We will argue that $B$ must be equal to $2$.

First suppose $B=1$.  Then $D$ must be either $2$ or $3$, since $D$ must be adjacent to $1$ and $6$.  If $D=3$, this contradicts the assumption that only one face is $Q_0$-supported. Hence $D=2$. If $D=2$, then $C$ must
be either $1$ or $6$. If $C=1$, we again have two $Q_0$-supported faces and hence a contradiction which implies $C=6$.  Now $A$ must be either $2$ or $3$, since it is a vertex adjacent to $1$ and $6$.  Either choice produces a second $Q_0$-supported face, and hence all cases lead to a contradiction, and thus $B\ne 1$. A similar (symmetrical) argument 
using the front face shows that $B\ne 6$. 

Continuing, suppose that $B=7$. The possible values for $A,C,D$, determined by adjacencies in $G$, are 
$D\in\{2, 10\}, A \in \{2,4\}, C \in \{1,6\}$. Note that the possibilities $C=3$ and $C=8$ are not included because neither is adjacent to either $2$ or $4$ in $G$.  We proceed systematically: If $D=10$, then $C=6$ since $10$ is not adjacent to $1$. This implies $A = 2$, since
$4$ is not adjacent to $6$, yielding a second $Q_0$-supported face, which is a contradiction and hence $D = 2$. If $C=1$, we obtain a second $Q_0$-supported face, hence $C=6$, which forces $A=2$ since $4$ is not adjacent to $6$. This choice produces a second $Q_0$-supported face, and hence a contradiction. This implies $B\ne 7$, which leaves $B=2$ as the only possible choice.

Since $2$ and $3$ are related by an automorphism of $G$, a symmetrical argument shows that $C=3$. From that point it is easy to show that $A=1$ and $D = 6$, thus forcing a second $Q_0$-supported face on top -- as well as a degenerate cube. This double contradiction completes the proof our first claim.

Next we claim that if the number of $Q_0$-supported faces, is $4$, then those faces must ``wrap around'' the $3$-cell, i.e. they avoid one of the three coordinate axes. If the number is $2$, the faces may either be adjacent (and ``hinged"), or opposite, in which case their labels differ by a 90 degree rotation.

We can conclude that if $Y$ is a $3$-cell, then the $Q_0$-supported faces can occur in arrangements of three types. If the number of such faces, is $4$, then those faces must ``wrap around'' the $3$-cell, i.e. they avoid one of the three coordinate axes.  This follows since if two $Q_0$-supported faces are dihedrally adjacent, then the two faces adjacent to both of them
cannot be $Q_0$-supported. If the number of $Q_0$-supported faces is $2$, the faces may either be adjacent (and ``hinged"), or opposite, in which case an easy argument shows that their labels differ by a 90 degree rotation.

With this information in hand, we can proceed to the proof of the main result.  
Suppose that $Y$ is a non-degenerate $3$-cell with two $Q_0$-supported faces $F_1$ and $F_2$ that are dihedrally adjacent. Let $\textrm{sgn}_\partial(F_1)$ and $\textrm{sgn}_\partial(F_2)$ denote the signs associated to $F_1$ and $F_2$ by the boundary operator $\partial_3$. Then we claim that
\begin{equation}\label{psi}
\chi(F_1)\, \textrm{sgn}_\partial(F_1) = - \chi(F_2)\, \textrm{sgn}_\partial(F_2).
\end{equation}
If we prove \ref{psi}, it will follow that $\Psi(\partial_3 Y) = 0$ in two of the three cases, i.e. either four $Q_0$-supported faces or two such faces that are dihedrally adjacent.
We will prove \ref{psi} in the form
\begin{equation}\label{psi2}
\chi(F_1) \chi(F_2) = - \textrm{sgn}_\partial(F_1) \textrm{sgn}_\partial(F_2).
\end{equation}
Since $\chi$ is a multiplicative character, we can regard $\chi(F_1) \chi(F_2)$ as $\chi(\sigma)$ where $\sigma$ is the
permutation mapping a generic set of (distinct) labels on the vertices of $F_1$ to the corresponding set of labels on
$F_2$ obtained by flipping 90 degrees through the dihedral edge. On both faces, the labels are read in standard order (recursively by dimension).

Denote the six faces of $Y$ by $F_{0**}, F_{1**}, F_{*0*}, F_{*1*},F_{**0}, F_{**1}$, with the obvious notation, e.g.,
$F_{*1*}$ denotes the face $\{(x,1,z)\}$, where $0 \le x,z \le 1$.  Let us say that $F$ is a {\em positive face} if $F$
is one of $F_{1**},  F_{*0*}, F_{**1}$, and a {\em negative face} if it is
one of $F_{0**},  F_{*1*}, F_{**0}$.  These designations correspond exactly to the signs of $\textrm{sgn}_\partial(F)$.
Hence \ref{psi2} can be interpreted as saying that
if $F_1$ and $F_2$ are dihedrally adjacent faces of a $3$-cell, both supported by $Q_0$, then
\begin{equation}\label{psi3}
\chi(F_1) \chi(F_2) = 
\begin{cases}
-1 & \textrm{if $F_1$ and $F_2$ are both positive or both negative}\\
+1 & \textrm{otherwise}\\
\end{cases}
\end{equation}

  For each dihedrally adjacent pair $F_1,F_2$, we can compute the left hand side of \ref{psi3} as the reflection character $\chi(\sigma)$ of the permutation $\sigma$ that maps the labeling of $F_1$ onto the labeling of $F_2$, where each labeling is read in the standard order.  There
are 12 cases (one for each edge), which can be grouped into four classes of permutations:  
\begin{itemize}
\item the identity $id$, with $\chi(id) = +1$,
\item a rotation $R$ through 90 degrees, with $\chi(R)=+1$,
\item a reflection $\phi$ around a diagonal axis, with $\chi(\phi) = -1$,
\item a reflection $\psi$ around an horizontal or vertical axis, with $\chi(\psi)=-1$.  
\end{itemize}
These four types are indicated on the edges of the following diagram. Edges in red correspond to permutations with 
$\chi(\sigma) = +1$ and edges in black correspond to permutations with $\chi(\sigma) = -1$.

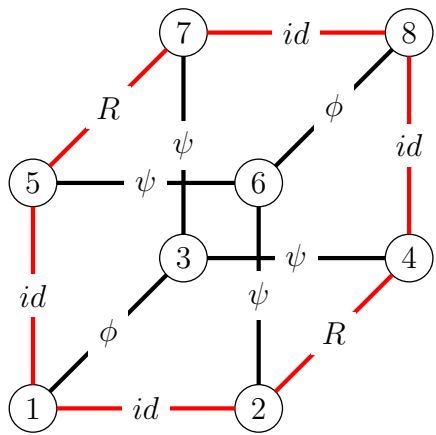
\begin{center}
\begin{tikzpicture}[scale=1]
   \Vertex[x=0 ,y=0,L=$1$]{1}
   \Vertex[x=3 ,y=0]{2}
   \Vertex[x=2,y=2,L=$3$]{3}
    \Vertex[x=5,y=2]{4}
      \Vertex[x=0,y=3,L=$5$]{5}
        \Vertex[x=3,y=3,L=$6$]{6}
          \Vertex[x=2,y=5,L=$7$]{7}
            \Vertex[x=5,y=5,L=$8$]{8}
   \Edge[style={ultra thick},color=red,label=$id$](1)(2)
    \Edge[style={ultra thick},color=black,label=$\phi$](1)(3)  \Edge[style={ultra thick},color=red,label=$R$](2)(4)  
    \Edge[style={ultra thick},color=black,label=$\psi$](3)(4)  
    \Edge[style={ultra thick},color=black,label=$\psi$](5)(6) \Edge[style={ultra thick},color=red,label=$R$](5)(7) \Edge[style={ultra thick},color=black,label=$\phi$](6)(8)
    \Edge[style={ultra thick},color=red,label=$id$](7)(8)
    \Edge[style={ultra thick},color=red,label=$id$](1)(5) \Edge[style={ultra thick},color=black,label=$\psi$](2)(6) \Edge[style={ultra thick},color=black,label=$\psi$](3)(7) \Edge[style={ultra thick},color=red,label=$id$](4)(8)
    \end{tikzpicture}
\end{center}

\noindent
The values of $\chi(\sigma)$ in each of the four cases can be verified in a straightforward manner.
Furthermore, it is easy to verify that for dihedral edges labeled in the diagram by $id$ or $R$, the corresponding pairs of faces $F_1, F_2$ have opposite boundary parity, i.e. $\textrm{sgn}_\partial(F_1) = - \textrm{sgn}_\partial(F_2)$, and for dihedral edges labeled by $\phi$ or $\psi$, the faces have the same boundary parity.  This is exactly the content of \ref{psi3}.

In order to complete the proof, it is only necessary to consider the case where $Y$ has exactly two opposite $Q_0$-supported faces $F_1$ and $F_2$,  with labels differing by a 90 degree rotation.  In this case $\chi(F_1) = \chi(F_2)$  and 
$\textrm{sgn}_\partial(F_1) = - \textrm{sgn}_\partial(F_2)$, implying \ref{psi} immediately.

We have  shown that $\Psi(\partial_3 Y) = 0$ for every $3$-cell $Y$. Since $\Psi(\theta) = 1$, where $\theta$ is defined in \ref{theta}, it follows that $\theta$ is not a boundary, and hence 
$\Hom^\Cube_2(G) \not\cong (0)$.  Further computation (not included here) shows that $\Hom^\Cube_2(G) \cong R$, i.e. the homology in dimension $2$ is generated by $\theta$.
\end{proof}

In fact, the same graph can be used to show that when $n\ge 3$, the
map $\psi_* : \Hom_n^\Cube(G) \rightarrow \Hom_n^\Path(G)$ defined \ref{sec-map} is not 
is not always surjective. Computations using \cite{Maerte17} for $\Hom^\Cube_\bullet$ and
Mathematica for $\Hom^\Path_\bullet$ (not displayed here) have shown that if $G$ is the graph from \ref{thm:counterexample} then
\[ \Hom_3^\Cube(G) \cong (0) \textrm{ and }\Hom_3^\Path(G) \cong R.
\]
It would be interesting to find self-contained, accessible proofs of these results.

\section{Comments, Examples, and Questions for Further Study\label{sec-questions}}

An important family of examples for which we have less than perfect information consists of the $k$-cycles 
$Z_k$, with $V(Z_k) = \{1,2,\dots, k\}$ and $E(Z_k) = \{ \{i,i+1\} \;|\; 1 \le i <k\} \cup \{\{k,1\}\}$. The following proposition states what we know about these graphs.

\begin{Proposition}\label{cycles}
Let $Z_k$ be a $k$-cycle. Then
\begin{enumerate}
    \item 
If $k=3$ or $k=4$, then 
$\Hom_n^\Cube(Z_k) \cong \Hom_n^\Path(Z_k) \cong (0)$ for all $n>0$.
\item
$\Hom_1^\Cube(Z_5) \cong R$,
$\Hom_2^\Cube(Z_5) \cong (0)$,
$\Hom_3^\Cube(Z_5) \cong (0)$.
\item
$\Hom_1^\Path(Z_k) \cong R$, and
$\Hom_n^\Path(Z_k) \cong (0)$ for $k>5,$ $n\ge 2$.
\end{enumerate}
\end{Proposition}
\begin{proof}
Since $Z_3$ is a complete graph and $Z_4$ is a $2$-cube, statement (1) follows from \Ref{cor-examples}, in both cases. 
In statement (2), the dimension 1 case follows by an easy direct calculation, or by applying Theorem 2.7 of \cite{BKLW} (see
also Theorem 5.2 in \cite{BBLL}) and Theorem 4.1 of \cite{BCW}.  In dimension 2 and 3, the results were obtained by computer computations which we do not include here.  In statement (3), the dimension 1 case again follows by an easy direct calculation, or by applying statement (2) together with \Ref{pr:yaubarcelo}. The remaining cases in statement (3) follow from a more general result stated in the next proposition.
\end{proof}

\begin{Proposition}\label{no4cycles}
Suppose that $G$ is an undirected graph containing no 3-cycles and no 4-cycles. Then 
$\Hom_n^\Path(G) \cong (0)$ for $n\ge 2$.
\end{Proposition}

\begin{proof}
If $n\ge 2$, then every generator of $\CC^\Path_n(G)$ must be an alternating path, that is,  it has the form $(a,b,a,b,\cdots ,b)$ or 
$(a,b,a,b,\cdots, a)$ for some pair $\{a,b\}\in E(G)$. Denote this path by $w_{ab}^n$, where $a$ and $b$ are the first two elements and the final element is determined by the parity of $n$. An easy computation shows that
\begin{equation}\label{wbdy}
\partial_n w^n_{ab}=
w^{n-1}_{ba} + (-1)^n w^{n-1}_{ab}
.
\end{equation}
Suppose that 
\[
\gamma = 
\sum_
{\{a,b\}\in E(G)}
c_{ab}\, w^n_{ab}
+ c_{ba} \,w^n_{ba}
\]
is a cycle in $\CC_n^\Path(G)$.
It follows from \Ref{wbdy} that
\[
c_{ba} + (-1)^{n} c_{ab} = 0
\]
for all $\{a,b\}\in E(G)$,  implying that $\gamma$ may be expressed as a linear combination of terms of the form
\[
w_{ba}^n + (-1)^{n+1} w_{ab}.
\]Since
$w_{ba}^{n+1} + (-1)^{n+1} w_{ab}
= \partial_{n+1} w^{n+1}_{ba}$, every cycle $\gamma\in \CC_n^\Path(G)$ is a boundary, and hence $\Hom_n^\Path(G) = (0).$
\end{proof}

We conjecture that \Ref{no4cycles} also holds for $\Hom_n^\Cube(G)$:

\begin{Conjecture}
Suppose that $G$ is an undirected graph containing no 3-cycles and no 4-cycles. Then 
$\Hom_n^\Cube(G) \cong (0)$ for $n\ge 2$.
\end{Conjecture}

Although our computational evidence is somewhat limited, it seems natural to conjecture that a  weaker property holds for all graphs $G$ and for both cubical and path homology.
\begin{Conjecture}
For any undirected graph $G$, there exists an integer $N$ such that we have $\Hom_n^\Cube(G)\cong \Hom_n^\Path(G) \cong (0)$ for $n\ge N$.
\end{Conjecture}

In cases where the methods of Sections 3 and 4 do not apply, computation of $\Hom_n^\Cube(G)$ remains a significant challenge, since the problem size increases 
rapidly with dimension. For example, when $G = Z_5$, $\rank \CC_3^\Cube(G) = 2230$ and
$\rank \CC_4^\Cube(G) = 978350$.  For the
graph $G$  in \ref{thm:counterexample}, 
$\rank \CC_3^\Cube(G) = 21552$ and
$\rank \CC_4^\Cube(G) = 21745744$. It would be useful to develop more effective tools to compute $\Hom^\Cube_n(G)$ for all $n$ in these cases.



\bibliographystyle{siam}
\bibliography{DiscHomo}

\begin{thebibliography}{10}

\bibitem{BBLL}
{\sc E.~Babson, H.~Barcelo, M.~de~Longueville, and R.~Laubenbacher}, {\em
  Homotopy theory of graphs}, Journal of Algebraic Combinatorics, 24 (2006),
  pp.~31--44.

\bibitem{BCW}
{\sc H.~Barcelo, V.~Capraro, and J.~A. White}, {\em Discrete homology theory
  for metric spaces}, Bull. London Math. Soc., 46 (2014), pp.~889--905.

\bibitem{BKLW}
{\sc H.~Barcelo, X.~Kramer, R.~Laubenbacher, and C.~Weaver}, {\em Foundations
  of connectivity theory for simplicial complexes}, Advances in Applied Math.,
  26 (2001), pp.~97--128.

\bibitem{Diestel00}
{\sc R.~Diestel}, {\em Graph Theory}, vol.~173 of Graduate Texts in
  Mathematics, Springer-Verlag, New York, 2000.

\bibitem{GLMY1}
{\sc A.~Grigor'yan, Y.~Lin, Y.~V. Muranov, and S.-T. Yau}, {\em Homologies of
  path complexes and digraphs}.
\newblock arXiv:1207.2834v4, 2013.

\bibitem{GLMY2}
\leavevmode\vrule height 2pt depth -1.6pt width 23pt, {\em Homotopy theory of
  diagraphs}, Pure and App. Math. Quat., 10 (2014), pp.~619--674.

\bibitem{GLMY3}
\leavevmode\vrule height 2pt depth -1.6pt width 23pt, {\em Cohomology of
  diagraphs and (undirected) graphss}, Asian J. Math., 19 (2015), pp.~887--931.

\bibitem{Maerte17}
{\sc J.~M\"arte}, {\em Discrete cubical homology of graphs}.
\newblock
  $\tt{https://github.com/jmaerte/discrete\underline{~}cubical\underline{~}\-}$
  \\ $\tt{homology\underline{~}of\underline{~}graphs}$, 2017.
\newblock Software Package.

\bibitem{Massey91}
{\sc W.~Massey}, {\em A Basic Course in Algebraic Topology}, vol.~91 of
  Graduate Texts in Mathematics, Springer, New York, 1991.

\bibitem{Mu}
{\sc J.~Munkres}, {\em Elements of Algebraic Topology}, Addison--Wesley
  Publishing Company, Reading, MA, 1984.

\bibitem{Hammack11}
{\sc S.~K. R.~Hammack, W.~Imrich}, {\em Handbook of Product Graphs}, Graduate
  Texts in Mathematics, CRC Press, Wellesley, 2011.

\bibitem{St}
{\sc R.~Stanley}, {\em Combinatorics and Commutative Algebra}, vol.~41 of
  Progress in Mathematics, Birkh{\"a}user, Boston, MA, 2~ed., 1996.

\end{thebibliography}
\end{document}